\documentclass[11pt,reqno]{amsart}

\usepackage[T1]{fontenc}
\usepackage{lmodern}
\usepackage{amsmath,amssymb,mathtools}
\usepackage{microtype}
\usepackage{comment} 
\usepackage[a4paper,margin=29mm]{geometry}
\usepackage[colorlinks=true,linkcolor=black,citecolor=blue,urlcolor=blue]{hyperref}
\usepackage[nameinlink,noabbrev]{cleveref}
\usepackage{xcolor} 
\numberwithin{equation}{section}
\allowdisplaybreaks[2]

\newtheorem{theorem}{Theorem}[section]
\newtheorem{proposition}[theorem]{Proposition}
\newtheorem{lemma}[theorem]{Lemma}
\newtheorem{corollary}[theorem]{Corollary}
\newtheorem{conjecture}[theorem]{Conjecture}
\theoremstyle{definition}

\theoremstyle{remark}
\newtheorem{remark}[theorem]{Remark}

\newcommand{\R}{\mathbb R}
\newcommand{\Sph}{\mathbb S}
\newcommand{\B}{\mathbb B}
\newcommand{\Bt}{\mathsf B}

\newcommand{\dA}{\,\mathrm dA}
\newcommand{\ds}{\,\mathrm ds}
\newcommand{\dbda}{\,\mathrm d\beta\,\mathrm d\alpha}
\newcommand{\Aop}{\mathcal A}
\newcommand{\Bop}{\mathcal B}
\newcommand{\Q}{\mathcal Q}
\newcommand{\Eone}{\mathcal E_1}
\newcommand{\Ccc}{\mathcal C_{\mathrm{crit}}}
\DeclareMathOperator{\diver}{div}
\DeclareMathOperator{\ind}{ind}
\DeclareMathOperator{\Span}{span}
\DeclareMathOperator{\Int}{Int}
\DeclareMathOperator{\sech}{sech}

\DeclareMathOperator{\tr}{tr}
\DeclareMathOperator{\artanh}{artanh}

\title[Uniqueness of free boundary minimal annuli]{Uniqueness of free boundary minimal annuli}

\author[Davide Parise]{Davide Parise}
\address{Department of Mathematics, Zeeman Building, University of Warwick, Gibbet Hill Road, Coventry CV4 7AL, UK}
\email{Davide.Parise@warwick.ac.uk}

\author[Jonathan J. Zhu]{Jonathan J. Zhu}
\address{Department of Mathematics, University of Washington, Seattle, WA, USA}
\email{jonozhu@uw.edu}

\date{\today}
\hypersetup{
  pdftitle={The critical catenoid and annular Bernoulli cones},
  pdfsubject={Uniqueness of free boundary minimal annuli},
  pdfkeywords={minimal annulus, critical catenoid, spherical Bernoulli problem, Steklov eigenvalue}
}

\begin{document}

\begin{abstract}
We prove that every smooth properly embedded free-boundary minimal annulus in the unit ball is congruent to the critical catenoid. We also classify positive solutions of $(\Delta_{\Sph^2}+2)u=0$ on smooth spherical annuli with $u=0$ and $|\nabla u|=1$ on the boundary; their one-homogeneous extensions are the axially symmetric Alt-Caffarelli cones. Lastly, we also treat the case of free-boundary minimal annuli in spherical caps. 
\end{abstract}

\maketitle

\section{Introduction}\label{sec:introduction}

Fraser and Li conjectured that the critical catenoid is the only properly embedded free-boundary minimal annulus in the Euclidean unit ball \cite{FL}. In this article, we confirm this conjecture. 

\begin{theorem}\label[theorem]{thm:main}
Let $\Sigma\subset\overline{\B^3}$ be a smooth properly embedded
free-boundary minimal annulus.  There is an orthogonal transformation
$O\in \mathrm{O}(3)$ such that $O\Sigma=\Ccc$.
\end{theorem}

The critical catenoid conjecture is often regarded as the free-boundary analogue of the Lawson conjecture (uniqueness of the Clifford torus in $\mathbb{S}^3$), which was proven by Brendle \cite{Brendle} using a two-point maximum principle. Previously, Fraser and Schoen proved the corresponding classification under the assumption that the coordinate functions belong to the first nonzero Steklov eigenspace \cite{FS}. In our approach we prove that every embedded free-boundary minimal annulus has this spectral property and that its eigenvalue-$1$ eigenspace consists precisely of the coordinate functions.

We also classify solutions of the overdetermined problem $(\Delta_{\Sph^2}+2)u=0$ with $u=0$ and $|\nabla u|=1$ on the boundary of a spherical annulus. In particular, we show that such solutions on an annular domain must be axially symmetric; see \Cref{thm:bernoulli}. The  relationship between this overdetermined problem and free-boundary (branched) minimal surfaces, via the dual surface map $X=\nabla u+up$, is classical; see \cite{EM,HKM}. For domains with the topology of an annulus, we are able to exclude branch points and prove global embeddedness of this dual surface, and then apply the catenoid classification. The annular hypothesis is essential: nonrotational homogeneous one-phase solutions with
higher-connectivity spherical links were constructed in \cite{HKM}.

The key geometric input for our classification of the critical catenoid is the radial-graph and boundary-convexity theorem of Kusner and McGrath \cite{KM}, which follows from the two-piece property of Lima and Menezes \cite{LM}; see also McGrath and Zou \cite{MZ}. We use it to choose an axis about which the angular projection has no critical points and winds once around each boundary component. A globally defined \textit{stream} function then identifies the annulus with a cylinder. In these coordinates, differentiation relates the transformed Steklov operator with Neumann boundary conditions to a Dirichlet operator admitting a positive nullfunction. Other early progress was made by, for instance, \cite{McG18, Seo1, Seo}. 

\subsection{The critical catenoid and first Steklov eigenfunctions}

Let $\B^3=\{x\in\R^3:|x|<1\}$ and $\Sph^2=\partial\B^3$. Throughout, a free-boundary minimal annulus is a smooth compact connected surface $\Sigma$ homeomorphic to $[0,1]\times\Sph^1$, together with a smooth immersion
\[
 X:\Sigma\longrightarrow\overline{\B^3},\qquad
 X(\Sigma^\circ)\subset\B^3,\qquad X(\partial\Sigma)\subset\Sph^2,
\]
which is minimal and meets $\Sph^2$ orthogonally.  The metric $g$ on $\Sigma$ is induced by $X$.  We write $\eta$ for the outward unit conormal and use the sign convention $\Delta=\diver\nabla$. For an embedding, we identify $\Sigma$ with its image.

Let $T>0$ be the unique solution of
\begin{equation}\label{eq:T}
 T\tanh T=1,
 \qquad c=\frac{1}{T\cosh T}.
\end{equation}
The critical catenoid is
\begin{equation}\label{eq:critical-param}
 \Ccc=
 \left\{c\bigl(\cosh s\cos\alpha,\cosh s\sin\alpha,s\bigr):
 -T\leq s\leq T,\quad \alpha\in\R/2\pi\mathbb Z\right\}.
\end{equation}

In \cite{FL}, Fraser--Li made the following conjecture:

\begin{conjecture}[Critical catenoid conjecture]
Every smooth properly embedded free-boundary minimal annulus in
$\overline{\B^3}$ is congruent to $\Ccc$.
\end{conjecture}

To prove \Cref{thm:main} and confirm the above conjecture, we will prove the following spectral result. 

\begin{theorem}\label[theorem]{thm:steklov}
Under the hypotheses of \Cref{thm:main},
\begin{equation}\label{eq:steklov-conclusion}
 \sigma_1(\Sigma)=\sigma_2(\Sigma)=\sigma_3(\Sigma)=1<\sigma_4(\Sigma),
 \qquad
 \Eone=\Span\{X_1,X_2,X_3\}.
\end{equation}
\end{theorem}
Note that we have indexed the Steklov eigenvalues with multiplicity as
$0=\sigma_0<\sigma_1\leq\sigma_2\leq\cdots$. To deduce \Cref{thm:main} from \Cref{thm:steklov} we use the following characterization by Fraser and Schoen:

\begin{theorem}[Fraser--Schoen, \cite{FS}, Theorem 1.2]\label[theorem]{thm:FS}
A free-boundary minimal annulus in the unit ball whose coordinate
functions are first nonzero Steklov eigenfunctions is congruent to the
critical catenoid.
\end{theorem}

\subsection{The spherical Bernoulli problem}\label{sec:statement}

Let $g_0$ be the round metric of curvature one on $\Sph^2\subset\R^3$,
and let $\Delta_0=\diver_{g_0}\nabla_0$. A smooth spherical annulus is
a connected open set $\Omega\subset\Sph^2$ whose closure is a smooth
compact annulus with two disjoint Jordan boundary curves.
We write $\nu$ for its outward $g_0$-unit conormal and assume
$u\in C^\infty(\overline\Omega)$.

Consider the overdetermined problem
\begin{equation}\label{eq:OEP}
 \begin{cases}
 (\Delta_0+2)u=0 & \text{in }\Omega,\\
 u>0 & \text{in }\Omega,\\
 u=0,\qquad |\nabla_0u|=1 & \text{on }\partial\Omega.
 \end{cases}
\end{equation}
The gradient normalisation is the same on both boundary components.
Since $u$ vanishes on the boundary and is positive in the interior,
\begin{equation}\label{eq:normal-data}
 \nabla_0u=-\nu,\qquad \partial_\nu u=-1
 \quad\text{on }\partial\Omega.
\end{equation}

With $T$ and $c$ as in \cref{eq:T}, set
\begin{equation}\label{eq:constants}
 s_*:=\tanh T=\frac1T,\qquad c=s_*\sqrt{1-s_*^2}.
\end{equation}

\begin{conjecture}[Annular spherical Bernoulli classification]
Every solution of \cref{eq:OEP} on a smooth spherical annulus is
rotationally symmetric.
\end{conjecture}

\begin{theorem}\label[theorem]{thm:bernoulli}
Let $\Omega\subset\Sph^2$ be a smooth spherical annulus, and let
$u\in C^\infty(\overline\Omega)$ satisfy \cref{eq:OEP}.
There is a unit vector $e\in\Sph^2$ such that
\begin{equation}\label{eq:classified-domain}
 \Omega=\{p\in\Sph^2:|p\cdot e|<s_*\},
\end{equation}
and, writing $s=p\cdot e$,
\begin{equation}\label{eq:classified-function}
 u(p)=c\bigl(1-s\,\artanh s\bigr).
\end{equation}
Conversely, \cref{eq:classified-domain,eq:classified-function} solve
\cref{eq:OEP}.
\end{theorem}

\subsection{Outline of the argument}

Our method centres ond the observation that cylindrical coordinates $(h,R,\alpha)$ in $\mathbb{R}^3$ (around an appropriate axis) satisfy good differential equations on a free-boundary minimal annulus $\Sigma$; certainly the height $h$ is a Steklov eigenfunction with eigenvalue 1. Moreover, after twisting these into suitable \textit{stream} coordinates $(\beta,\alpha)$ on the annulus, the $\beta$ derivative of the known eigenfunction depends upon the support function on $\Sigma$, which is known to be positive. This allows us to characterise solutions of the equation satisfied by $\beta$-derivatives of eigenfunctions. In turn, implies that negative eigenfunctions may only depend on the $\alpha$, and the remaining circle operator has just one negative mode.

More concretely, in the stream coordinates $(\beta,\alpha)$, multiplication by the
cylindrical radius $R$ transforms the (modified) energy form into
\[
\Q_\Sigma(Rv)
=\int_C\bigl(R^4 v_\beta^2+v_\alpha^2-v^2\bigr)\dbda,
\]
which has associated (Neumann) operator $\mathcal{A}$. The quotient $t=h/R$ satisfies $\Aop t=0$ with Neumann data, and $t_\beta$ is proportional to the support function and hence positive. Defining a Dirichlet operator $\mathcal{B}$ by $\partial_\beta\Aop v=\Bop(v_\beta)$, we see that $t_\beta$ is positive Dirichlet 0-eigenfunction for $\Bop$, so effectively bounds the eigenvalues of $\mathcal{B}$ from below via a comparison principle argument. It follows that every negative
Neumann eigenfunction of $\Aop$ is independent of $\beta$.
The remaining circle operator is simply $\frac{d^2}{d\alpha^2}+1$, which has eigenvalues $n^2-1$, $n\in\mathbb{Z}$. It follows that $\ind\Q_\Sigma=1$ and hence $\sigma_1=1$.

The geometric and spectral parts of the proof occupy
\cref{sec:geometry,sec:stream,sec:operator,sec:spectral,sec:conclusion};
the Bernoulli dual and classification are proved in
\cref{sec:bernoulli}. 

\subsection{Spherical caps and the closed case} 

We conclude this introduction by comparing our methods in the free-boundary setting to the analogous closed setting, that is, the Lawson conjecture. In fact, in \cref{sec:spherical-caps} we describe how our methods may be used to obtain rotational symmetry for embedded free-boundary minimal annuli in spherical caps. In particular, when $\vartheta=\pi/2$, \cref{thm:cap-rigidity} gives that embedded free-boundary minimal annuli in the hemisphere must be rotationally symmetric, and hence half of the Clifford torus. That is, our method yields that any closed embedded minimal torus in $\mathbb{S}^3$ \textit{with a reflection symmetry} must be the Clifford torus. 

If one attempts to use stream coordinates on embedded minimal tori $\Sigma$ in $\mathbb{S}^3$, the formal arguments appear to be promising. However, the main obstacle seems to be finding a rotational `axis' (in this setting, a great circle) which is avoided by $\Sigma$. On the other hand, from the authors' own investigations and conversations with other experts, the naive attempt to use Brendle's proof \cite{Brendle} of the Lawson conjecture appears to fail in the free-boundary setting (except for the hemisphere). The issue is that the `ball curvature' does not seem to satisfy a sufficiently helpful inequality at the boundary. 

\begin{remark}
    While preparing this manuscript, we were made aware of the work of Chodosh and Gianocca \cite{CG}, who have proven that embedded free-boundary minimal surfaces in $\overline{\mathbb{B}^3}$ of genus zero (in particular, annuli) are embedded by first Steklov eigenfunctions. Their method uses a very clever representation of the index form by a radial factor and a nonnegative divergence free tensor field, while ours uses stream coordinates relative to an axis and is so adapted to annuli. Our proofs are distinct and have been obtained simultaneously and independently. 
\end{remark}

\subsection*{Acknowledgements} The authors would like to thank Otis Chodosh and Matilde Gianocca for graceful and helpful exchanges regarding their work \cite{CG}, as well as Bill Minicozzi and Nick Edelen for their advice and encouragement. J.Z. was supported in part by the National Science Foundation under grant DMS-2439945.

\subsection*{AI Declaration Statement} ChatGPT was used in the exploratory phase of this project to prune nonviable solution paths and was later used to check the validity of some routine arguments and grammar. The mathematical approach in this paper was not generated by an LLM, nor was an LLM used to generate the text of the paper. 

\section{Preliminaries and notation}\label{sec:preliminaries}
\subsection{The modified energy form}
For real-valued $f,h\in H^1(\Sigma)$, set
\begin{equation}\label{eq:Q}
 \Q_\Sigma(f,h)
 =\int_\Sigma\langle\nabla f,\nabla h\rangle\dA
   -\int_{\partial\Sigma}fh\ds,
\end{equation}
and write $\Q_\Sigma(f)=\Q_\Sigma(f,f)$.
The boundary term is understood in the trace sense and is continuous
on $H^1(\Sigma)\times H^1(\Sigma)$ by the trace inequality.
The index of $\Q_\Sigma$ is the maximal dimension of a subspace on which
$\Q_\Sigma$ is negative definite.  This is the index of the modified energy form, not the Morse index of the area functional.
The Steklov eigenvalue problem is
\begin{equation}\label{eq:steklov-problem}
 \Delta f=0\quad\text{in }\Sigma^\circ,
 \qquad \partial_\eta f=\sigma f\quad\text{on }\partial\Sigma.
\end{equation}
Here $\Sigma^\circ=\Sigma\setminus\partial\Sigma$.
All Sobolev spaces and quadratic forms are real. Throughout,
$C^\infty$ on a compact manifold with boundary means
smooth up to the boundary. We use $\dA$ and $\ds$ for the Riemannian
area and arclength measures with respect to the induced metric $g$. When an orientation is fixed, we also use $\dA_g$ to denote the associated area two-form; the Hodge star is specified by $\zeta\wedge\star\vartheta=\langle\zeta,\vartheta\rangle_g \dA_g$.

We summarise some basic properties of free-boundary minimal immersions in the following lemma. 

\begin{lemma}\label[lemma]{lem:coordinate-identities}
For a free-boundary minimal immersion,
\begin{equation}\label{eq:coordinate-identities}
 \Delta X=0,\qquad \partial_\eta X=X,\qquad
 \Delta|X|^2=4.
\end{equation}
In particular, $|X|<1$ in $\Sigma^\circ$ and
\begin{equation}\label{eq:flux-balance}
 \int_{\partial\Sigma}X\ds=0.
\end{equation}
In particular, every nonzero coordinate function is a Steklov eigenfunction of eigenvalue $1$.
\end{lemma}
\begin{proof}
    The identities \eqref{eq:coordinate-identities} are standard calculations for free-boundary minimal immersions. The inequality $|X|<1$ on the interior then follows from the maximum principle on $|X|^2$, and \eqref{eq:flux-balance} follows by integrating $\Delta X$ by parts.
\end{proof}

\begin{lemma}[Index criterion]\label[lemma]{lem:index-criterion}
If a free-boundary minimal annulus satisfies $\ind\Q_\Sigma=1$, then
$\sigma_1(\Sigma)=1$.
\end{lemma}

\begin{proof}
At least one coordinate function is nonconstant, so
$\sigma_1\leq1$ by \cref{lem:coordinate-identities}.
Suppose $0<\sigma_1<1$, and let $f\not\equiv0$ be a corresponding
Steklov eigenfunction so that $\Delta f =0$, $\partial_\eta f = \sigma_1 f|_{\partial\Sigma}$. Integrating by parts gives
\[
 0=\int_{\partial\Sigma}\partial_\eta f\ds
   =\sigma_1\int_{\partial\Sigma}f\ds.
\]
Consequently,
\[
 \Q_\Sigma(f)=(\sigma_1-1)\int_{\partial\Sigma}f^2\ds<0,
 \qquad
 \Q_\Sigma(1)=-|\partial\Sigma|<0,
 \qquad \Q_\Sigma(1,f)=0.
\]
Note that the trace $f|_{\partial\Sigma
}$ is nonzero, as otherwise harmonicity would give
$f=0$.  Thus $\Q_\Sigma$ is negative definite on
$\Span\{1,f\}$, contradicting the index assumption.
\end{proof}

\section{Finding the flux axis and the angular coordinate}
\label{sec:geometry}

In this section, we define the angular component of our \textit{stream} coordinates by finding a suitable axis and restricting the corresponding cylindrical coordinates. We will crucially use certain radial graph properties which are known for embedded free-boundary minimal surfaces of genus zero. The fact that there are precisely 2 boundary components will be essential in choosing the axis via the balancing condition \eqref{eq:flux-balance}. We summarise the radial graph properties for annuli as follows:

\begin{proposition}[Radial graph and boundary convexity]
\label[proposition]{thm:radial}
Let $\Sigma\subset\overline{\B^3}$ be as in \cref{thm:main}.
Then $0\notin\Sigma$, and the radial projection
\[
 \pi:\Sigma\rightarrow\Sph^2,\qquad \pi(x)=\frac{x}{|x|},
\]
is a homeomorphism onto the closure of a spherical annulus
$\Omega_X$, and a diffeomorphism from $\Sigma^\circ$ onto $\Omega_X$.
Writing $\partial\Sigma=\Gamma_+\sqcup\Gamma_-$, there are disjoint
closed geodesically convex discs $D_+,D_-\subset\Sph^2$, each contained
in an open hemisphere, such that
\begin{equation}\label{eq:radial-domain}
 \partial D_\pm=\Gamma_\pm,
 \qquad \Omega_X=\Sph^2\setminus(D_+\cup D_-).
\end{equation}
\end{proposition}

The radial-graph and convexity conclusions are
\cite[Corollary 4.2]{KM}; see also \cite[Lemma 2.1 and its proof]{MZ} for the global homeomorphism and the complementary convex discs. The spherical convexity facts, including containment in an open hemisphere, are recorded in \cite[Section~3 and Lemma~4.2 of the arXiv version]{Seo}. Note that the radial graph property is, at least locally, equivalent to the following positivity of the support function. 

\begin{lemma}\label[lemma]{lem:support}
There is a smooth unit normal $N$ on $\Sigma$ such that the support function
\begin{equation}\label{eq:positive-support}
 \xi:=X\cdot N>0\quad\text{in }\Sigma^\circ,
 \qquad \xi=0\quad\text{on }\partial\Sigma.
\end{equation}
\end{lemma}

One may also prove the above lemma directly using, for instance, the two-piece property as in \cite[Proposition 3.12]{NZcaps}.

Let $e\in\mathbb{S}^2$ and consider the axis $\mathbb{R}e$. Then we have cylindrical coordinates on $\mathbb{R}^3$ defined by 
\begin{equation}\label{eq:cylindrical}
h=X\cdot e,\qquad R=|X-he|,
\qquad X=R e_r(\alpha)+he,
\end{equation}
where 
\[
e_r(\theta)\coloneqq\cos\theta\,e_1+\sin\theta\,e_2
\]
and $(e_1,e_2,e)$ is a fixed oriented orthonormal basis.

We may consider the angular projection map $\mathbb{R}^3\to S(e^\perp)$, where $S(e^\perp)$ is the unit circle in the 2-plane $e^\perp$, given explicitly (for $x\notin \mathbb{R}e$) by
\begin{equation} \label{eq:angular-projection}
    \mathcal P_e(x) = \frac{x-(e\cdot x)e}{\sqrt{|x|^2-(e\cdot x)^2}} \in S(e^\perp),
\end{equation}

First, we establish that $\alpha$ will give a suitable angular coordinate on the boundary, which essentially relies upon convexity of the boundary.

\begin{lemma}\label[lemma]{lem:cap-flux}
Let $D\subset\mathbb S^2$ be a closed geodesically convex disc
contained in an open hemisphere, with smooth Jordan boundary
$\Gamma$. Then
\[
F:=\int_\Gamma x\,\mathrm{d}s\neq0,
\qquad
e_D:=\frac{F}{|F|}\in\operatorname{Int}D.
\]
For every $e\in\operatorname{Int}D$, angular projection about the
axis $\mathbb{R}e$ restricts to a diffeomorphism
$\Gamma\to\mathbb R/2\pi\mathbb Z$.
\end{lemma}

\begin{proof}
Choose $p\in\mathbb S^2$ such that $p\cdot x>0$ on $D$.
Compactness and geodesic convexity imply that
\[
K:=\{rx:r\geq0,\ x\in D\}
\]
is a closed convex cone. Indeed, the direction of a positive
linear combination of two points of $D$ lies on their shorter
great-circle segment. Polar coordinates also give
\[
\operatorname{Int}_{\mathbb R^3}K
=
\{rx:r>0,\ x\in\operatorname{Int}_{\mathbb S^2}D\}.
\]

Since $p\cdot F=\int_\Gamma p\cdot x\,\mathrm{d}s>0$, we have
$F\neq0$. Positive Riemann sums and closedness give $F\in K$.
If $F\in\partial K$, the supporting-hyperplane theorem yields a
nonzero linear functional $\ell$ with
\[
\ell\geq0\quad\text{on }K,
\qquad
\ell(F)=0.
\]
The supporting hyperplane passes through the origin because
$K$ is a cone: applying the supporting inequality at $0$ and
$2F$ forces $\ell(F)=0$. Since $\Gamma$ is a Jordan curve
contained in an open hemisphere, it cannot be contained in a
great circle. Thus $\ell$ is positive on a nonempty open arc
of $\Gamma$, and
\[
0=\ell(F)=\int_\Gamma\ell(x)\,\mathrm{d}s>0,
\]
which is a contradiction. Hence $F\in\operatorname{Int}K$, and so 
$e_D\in\operatorname{Int}D$.

Now fix $e\in\operatorname{Int}D$. Recall that a geodesically convex disc is star-shaped (about any interior point $e$). For $v\in S(e^\perp)$, consider the oriented meridian $\gamma_v(t)=\cos t\,e+\sin t\,v$ for $ 0\leq t\leq\pi.$ Convexity, $e\in\Int D$, and $-e\notin D$ imply $\gamma_v^{-1}(D)=[0,\tau_v]$ for some $0<\tau_v<\pi$, with $\gamma_v([0,\tau_v))\subset\Int D$. Thus each such meridian meets $\Gamma$ exactly once, and $\mathcal P_e|_\Gamma$ is bijective.

It remains to check transversality. Let $\nu_D$ be the inward unit conormal to $\Gamma$ in $\mathbb S^2$. At $x\in\Gamma$, the angular projection is given by  
\[\mathcal P_e(x) = \frac{x-(e\cdot x)e}{\sqrt{1-(e\cdot x)^2}}.\]
Moreover, the tangent plane to $\partial K$ at $x$ is a supporting plane, so that $\nu_D(x)\cdot y\geq0$ for $y\in K.$ Since $e\in\operatorname{Int}K$, this inequality is strict at $e$: otherwise a small perturbation of $e$ in the direction $-\nu_D(x)$ would contradict the supporting inequality. Orient the unit tangent $T$ on $\Gamma$ so that $x\times T=\nu_D(x)$, and identify $S(e^\perp)$ with $\mathbb R/2\pi\mathbb Z$ using the orientation induced by $e$. Differentiating the formula for $\mathcal P_e(x)$, the corresponding angular differential satisfies
\[
d\alpha_x(T) = \frac{(e\times x)\cdot T}{1-(e\cdot x)^2} = \frac{e\cdot\nu_D(x)}{1-(e\cdot x)^2}>0.
\]
The differential of $\mathcal P_e|_\Gamma$ is therefore nonzero
at every point. The inverse function theorem and the preceding
bijectivity then imply that $\mathcal P_e|_\Gamma:\Gamma\to S(e^\perp)$ is a diffeomorphism.
\end{proof}

We now use the balancing condition \eqref{eq:flux-balance} to define an axis that avoids $\Sigma$, and restrict the cylindrical coordinates \eqref{eq:cylindrical} to the surface. By slight abuse of notation we will continue to denote these by $h,R,\alpha$. In particular, $\alpha$ will form a good angular coordinate on the surface, with nonvanishing differential (including on the boundary).

\begin{proposition}[An angular submersion of period $2\pi$]
\label[proposition]{prop:angle}
Let $\Sigma$ be as in \cref{thm:main}. 
There is a unit vector $e\in\Sph^2$ such that
$X(\Sigma)\cap\R e=\varnothing$. 

Moreover, fixing an oriented orthonormal basis
$(e_1,e_2,e)$, the cylindrical coordinates 
\[h=X\cdot e,\qquad R=|X-he|,
\qquad X=R e_r(\alpha)+he,\]
define smooth real-valued functions $h,R$ on $\Sigma$ and a smooth
circle-valued map
$\alpha:\Sigma\to\R/2\pi\mathbb Z$ such that
\begin{equation}\label{eq:angular-properties}
R>0\quad\text{on }\Sigma,
\qquad d\alpha\neq0\quad\text{on }\Sigma.
\end{equation}
Finally, each restriction
$\alpha|_{\Gamma_\pm}:\Gamma_\pm\to\R/2\pi\mathbb Z$
is a diffeomorphism.
\end{proposition}

\begin{proof}
Let
\[
F_\pm:=\int_{\Gamma_\pm}X\ds.
\]
By the balancing \cref{eq:flux-balance}, $F_-=-F_+$. Applying \cref{lem:cap-flux} to the
geodesically convex discs $D_\pm$ gives
\[
F_+\neq0,\qquad e:=\frac{F_+}{|F_+|}\in\Int D_+,\qquad -e=\frac{F_-}{|F_-|}\in\Int D_-.
\]
The radial projection is defined on the whole annulus and,
by \eqref{eq:radial-domain}, its image contains neither $e$
nor $-e$. Therefore $X(\Sigma)\cap\R
e=\varnothing$, and $R>0$ everywhere, including on the boundary.

Now write $x_j=X\cdot e_j$ for $j=1,2$, so that $R^2=x_1^2+x_2^2$. Since $R>0$, the formula
\[
e^{i\alpha}=\frac{x_1+ix_2}{R}
\]
defines a smooth circle-valued map on $\Sigma$. Locally one may choose a smooth
real-valued lift of $\alpha$ and write $x_1=R\cos\alpha$ and $x_2=R\sin\alpha$.
Differentiating these identities gives
\begin{equation}\label{eq:dalpha}
d\alpha=\frac{x_1\,dx_2-x_2\,dx_1}{R^2}.
\end{equation}
Any two local lifts differ by a locally constant integer multiple of $2\pi$, so
their differentials agree. Consequently, \cref{eq:dalpha} defines a global smooth
one-form (without requiring a globally real-valued angle). Along $\Sigma$, set
$e_\theta=e\times e_r(\alpha) =-\sin\alpha\,e_1+\cos\alpha\,e_2.$ The decomposition
$X=R e_r(\alpha)+he$ implies that $X\cdot e_\theta=0.$ Moreover, \cref{eq:dalpha}
gives, for $q\in\Sigma$ and $v\in T_q\Sigma$,
\[
d\alpha_q(v) = \frac{\langle e_\theta(q),dX_q(v)\rangle}{R(q)}.
\]
Suppose that $d\alpha_q=0$ at some $q\in\Sigma^\circ$. Then
$e_\theta(q)$ is perpendicular to $dX_q(T_q\Sigma)$. Since this tangent plane has
codimension one and $e_\theta(q)$ is a unit vector, the normal from
\cref{lem:support} satisfies $N(q)=\pm e_\theta(q)$, and so
\[
X(q)\cdot N(q) = \pm X(q)\cdot e_\theta(q)=0.
\]
Since we had $X\cdot N>0$ in the interior, we conclude that $d\alpha\neq0$ on $\Sigma^\circ$.

For $x\in\partial\Sigma \subset \mathbb{S}^2\setminus\{e,-e\}$, the angular projection 
\eqref{eq:angular-projection} is precisely \[\mathcal
P_e(X)=\mathcal
P_{-e}(X)=e_r(\alpha).\] Since $e\in\Int D_+$, \cref{lem:cap-flux} shows that $\mathcal
P_e|_{\Gamma_+}$ is a diffeomorphism onto $S(e^\perp)$. Applying the same lemma to
$D_-$ with pole $-e$ gives the corresponding conclusion for $\Gamma_-$.

Identifying $S(e^\perp)$ with $\R/2\pi\mathbb Z$ via
$\theta\mapsto e_r(\theta)$, we conclude that both
$\alpha|_{\Gamma_\pm}:\Gamma_\pm \to \R/2\pi\mathbb Z$ are diffeomorphisms. In particular,
\[
d\alpha_q(T)\neq0
\qquad
\text{for every }q\in\partial\Sigma
\text{ and }0\neq T\in T_q\partial\Sigma.
\]
This establishes the boundary nonvanishing in
\eqref{eq:angular-properties}. Notice that it is the
\emph{tangential} component that is nonzero: the free-boundary
condition $dX(\eta)=X$ gives
\[
d\alpha(\eta) = \frac{\langle e_\theta,dX(\eta)\rangle}{R} = \frac{\langle e_\theta,X\rangle}{R}=0.
\]
\end{proof}

Note that the boundary diffeomorphisms above imply
\[
\left|\int_{\Gamma_\pm}d\alpha\right|=2\pi.
\]
Since either boundary component generates $H_1(\Sigma;\mathbb Z)$, the angular map
has winding number $\pm1$ on a generator of the annulus. 

\section{Global stream coordinates}
\label{sec:stream}

In this section, we will complete our global \textit{stream} coordinates by finding a coordinate $\beta$ conjugate to the angular coordinate $\alpha$. In what follows, we fix the axis $\mathbb{R}e$ from \cref{prop:angle} and an orthonormal frame $(e_1,e_2,e_3:=e)$, so that the ambient coordinates $X_1=R\cos\alpha$, $X_2=R\sin\alpha$, and $X_3=h$. Also set
\begin{equation}\label{eq:G}
 G:=|\nabla\alpha|^2>0.
\end{equation}
Here the one-form $d\alpha$ is given by \cref{eq:dalpha}, and
$\nabla\alpha$ denotes its metric-dual vector field. (A real-valued angle is only used locally.)

First, we verify the fundamental equations satisfied by the cylindrical coordinates $R,\alpha$. Intuitively, one should expect a (weighted) harmonicity for the angular coordinate $\alpha$, since rotation generates a Killing field on $\mathbb{R}^3$ and $\Sigma$ is minimal. Similarly, the distance $R$ to the axis should inherit good equations from the defining equations for $X$; using its strict positivity on $\Sigma$ will later give a helpful reparametrisation of the quadratic form $\mathcal{Q}_\Sigma$.

\begin{lemma}\label[lemma]{lem:weighted-angle}
In $\Sigma$, we have
\begin{equation}\label{eq:weighted-angle}
 \Delta R=RG,
 \qquad \diver(R^2\nabla\alpha)=0,
\end{equation}
while on $\partial\Sigma$,
\begin{equation}\label{eq:R-alpha-boundary}
 \partial_\eta R=R,
 \qquad \partial_\eta\alpha=0.
\end{equation}
\end{lemma}

\begin{proof}
On each sufficiently small neighbourhood, choose a smooth real-valued lift of the angular map. Since the coordinate functions are harmonic,
\[
\begin{aligned}
0&=e^{-i\alpha}\Delta(Re^{i\alpha})\\
 &=\Delta R-R|\nabla\alpha|^2
 +i\bigl(2\langle\nabla R,\nabla\alpha\rangle+R\Delta\alpha\bigr).
\end{aligned}
\]
The real part gives $\Delta R=RG$. Multiplying the imaginary part by $R$ gives $\diver(R^2\nabla\alpha)=0$. These identities agree on overlaps because any two local lifts differ by a constant multiple of $2\pi$. Similarly,
\[
e^{-i\alpha}\partial_\eta(Re^{i\alpha})
=\partial_\eta R+iR\partial_\eta\alpha=R
\]
gives both boundary identities.
\end{proof}

Note that the coordinate function $h= X\cdot e$ is already known to satisfy $\Delta h=0$, $\partial_\eta h =h|_{\partial\Sigma}$. We orient $\Sigma$ by the unit normal $N$ chosen in \cref{lem:support}, and use the corresponding Hodge star. 

\begin{proposition}[Global stream cylinder]
\label[proposition]{prop:stream-cylinder}
Let $\Sigma$ be as in \cref{thm:main}. 
There are a smooth real-valued function $\beta$ on $\Sigma$
and numbers $b_-<b_+$, such that
\begin{equation}\label{eq:stream-form}
d\beta= R^2\star d\alpha,
\end{equation}
and
\begin{equation}\label{eq:stream-diffeo}
(\beta,\alpha):\Sigma \rightarrow C \coloneqq [b_-,b_+]\times(\R/2\pi\mathbb Z)
\end{equation}
is a diffeomorphism of compact manifolds with boundary. Here $\star$ is the Hodge star for a choice of orientation of $(\Sigma,g)$. 

\end{proposition}

\begin{proof}
Set $\omega:=R^2\star d\alpha$. By \cref{eq:weighted-angle},
\[
d\omega=\diver(R^2\nabla\alpha)\dA_g=0.
\]  
Along each boundary component, choose the unit tangent $T$ so that $(\eta,T)$ is positively oriented, where $\eta$ is the outward unit conormal. Then
\[
\star\eta^\flat=T^\flat, \qquad \star T^\flat=-\eta^\flat,
\]
and consequently
\[
\omega(T)=R^2(\star d\alpha)(T) = R^2d\alpha(\eta)=0
\]
by \cref{eq:R-alpha-boundary}. Thus, for the inclusion
$\iota:\partial\Sigma\hookrightarrow\Sigma$, we have $\iota^*\omega=0.$

Since $\Sigma$ is an annulus, either boundary component generates
$H_1(\Sigma;\mathbb Z)$. Since $\omega$ is closed and its integral around a boundary generator vanishes, $\omega$ has zero period on every closed curve. Therefore $\omega$ is exact, and is induced by a smooth function $\beta$ as in \cref{eq:stream-form}.

(Concretely, fixing $q_0\in\Sigma$, we
may define the primitive
\[
\beta(q):=\int_{q_0}^{q}\omega,
\]
where the integral is independent of the chosen path. This gives a smooth
real-valued function, up to an additive constant, satisfying \cref{eq:stream-form}.
Smoothness up to the boundary follows from the local primitive construction in
boundary coordinate half-discs.) 

Since $d\beta(T)=\omega(T)=0$, the function $\beta$ is constant on each boundary component.
Moreover, the Hodge star is an isometry on one-forms and rotates them orthogonally,
so
\begin{equation}\label{eq:stream-gradients}
\langle\nabla\beta,\nabla\alpha\rangle=0, \qquad |\nabla\beta|^2=R^4|\nabla\alpha|^2 = R^4 G>0
\quad\text{on }\Sigma.
\end{equation}
In particular, $\beta$ is nonconstant and has no interior critical points.
Compactness therefore gives distinct minimum and maximum values, both attained on
the boundary. As $\beta$ is constant on each of the two boundary components, these
constants are precisely its minimum $b_-$ and maximum $b_+$. Thus, $b_-<\beta<b_+$
in $\Sigma^\circ$. Write
\[
\Gamma_{\mathrm{lo}}:=\{\beta=b_-\}, \qquad \Gamma_{\mathrm{hi}}:=\{\beta=b_+\}
\]
for the two boundary components, labelled by their $\beta$-values.

Consider the smooth vector field
\[
V:=\frac{\nabla\beta}{|\nabla\beta|^2}.
\]
By \cref{eq:stream-gradients}, $d\beta(V)=1$ and $d\alpha(V)=0$. On the boundary,
$\nabla\beta$ is conormal because $\beta$ is constant along each component. Extremality of the boundary values for $\beta$ then implies
\[
\partial_\eta\beta<0\quad\text{on }\Gamma_{\mathrm{lo}},
\qquad
\partial_\eta\beta>0\quad\text{on }\Gamma_{\mathrm{hi}}.
\]
That is, $V$ points into $\Sigma$ along $\Gamma_{\mathrm{lo}}$ and out of $\Sigma$
along $\Gamma_{\mathrm{hi}}$.

Extend $V$ smoothly to a neighbourhood of $\Sigma$ in its double and denote its flow
by $\Theta_s$. 
For $q\in\Gamma_{\mathrm{lo}}$, as long as the
trajectory remains in $\Sigma$,
\[
\beta(\Theta_s(q))=b_-+s, \qquad \alpha(\Theta_s(q))=\alpha(q).
\]
Compactness and smooth continuation of the flow imply that this trajectory exists
for $0\leq s\leq b_+-b_-$ and remains in $\Sigma$ throughout that interval: it
enters the interior initially and cannot meet either boundary at a time
$0<s<b_+-b_-$. At time $b_+-b_-$ it reaches $\Gamma_{\mathrm{hi}}$.

Conversely, flowing backwards from any $x\in\Sigma$ reaches $\Gamma_{\mathrm{lo}}$
after time $\beta(x)-b_-$. Uniqueness of integral curves therefore
shows that
\[
\Psi:[b_-,b_+]\times\Gamma_{\mathrm{lo}} \rightarrow\Sigma, \qquad \Psi(s,q):=\Theta_{s-b_-}(q),
\]
is bijective. Both $\Psi$ and its inverse
\[
\Psi^{-1}(x) = \bigl(\beta(x),\Theta_{b_--\beta(x)}(x)\bigr)
\]
are smooth up to the boundary, by smooth dependence of the flow on time and initial data.

Finally, $\alpha$ is constant along the flow trajectories, and
$\alpha|_{\Gamma_{\mathrm{lo}}}$ is a diffeomorphism onto $\R/2\pi\mathbb Z$ by
\cref{prop:angle}. Consequently, $(\beta,\alpha)\circ\Psi(s,q)=(s,\alpha(q))$, which
proves \cref{eq:stream-diffeo}. Replacing $\beta$ by
$-\beta$ reverses $\varepsilon$ and replaces the interval
$[b_-,b_+]$ by $[-b_+,-b_-]$.
\end{proof}

\begin{lemma}[Metric and boundary measures]\label[lemma]{lem:metric}
In the coordinates of \cref{prop:stream-cylinder},
\begin{equation}\label{eq:metric}
 g=\frac{1}{R^4G}\,d\beta^2+\frac{1}{G}\,d\alpha^2,
 \qquad dA=\frac{1}{R^2G}\,d\beta\,d\alpha.
\end{equation}
The area formula in \cref{eq:metric} is an identity of positive
measures, independent of the orientation of the coordinates.
At the two cylinder boundaries,
\begin{equation}\label{eq:boundary-metric}
 \eta\big|_{\beta=b_\pm}
 =\pm R^2\sqrt G\,\partial_\beta,
 \qquad ds=G^{-1/2}\,|d\alpha|.
\end{equation}
In particular, $R$ and $G$ are smooth and bounded above and below by
positive constants on $C$.
\end{lemma}

\begin{proof}
As in \cref{eq:stream-gradients}, the inverse metric satisfies $g^{\beta\beta}=|\nabla\beta|^2=R^4G$,
$g^{\alpha\alpha}=G$, and
$g^{\beta\alpha}=\langle\nabla\beta,\nabla\alpha\rangle=0$. Inverting gives the
metric, and taking its determinant gives the area measure.  Normalising
$\partial_\beta$ gives the conormal; restricting the metric to the boundary gives
its arclength measure. The final assertion follows from smoothness, positivity, and
compactness.
\end{proof}

\begin{remark}
Radial projection is singular along the free boundary, because
$X=\eta\in T\Sigma$ there. The stream coordinates are instead regular:
\cref{eq:stream-gradients} holds up to the boundary.  No boundary
regularity of the inverse radial parametrisation is used below.
\end{remark}

\section{A monotone kernel function}
\label{sec:operator}

In this section, we rewrite the modified energy form $\Q_\Sigma$ in our stream coordinates $(\beta,\alpha)$. 

Concretely, we will work on the cylinder $C=[b_-,b_+]\times(\R/2\pi\mathbb Z)$, equipped with the flat product metric $g_{\mathrm{flat}}=d\beta^2+d\alpha^2$. Later, we will consider the corresponding Sobolev
space $H^1(C)$, with norm
\[
\|u\|_{H^1(C)}^2 = \int_C \left(|u|^2+|u_\beta|^2+|u_\alpha|^2\right)\dbda.
\]
Equivalently, $H^1(C)$ consists of $H^1$ functions on the cut rectangle $[b_-,b_+]\times [0,2\pi]$ whose boundary traces agree at $\alpha=0$ and $\alpha=2\pi$ (no boundary condition is imposed at $\beta=b_\pm$). Recall that $C^\infty(C)$ (understood as smooth up to the boundary) is dense in $H^1(C)$.

By a slight abuse of notation, we will identify functions on $C$ with their pullbacks via the coordinate diffeomorphism $\Phi=(\beta,\alpha):\Sigma\to C$.

\begin{proposition}\label[proposition]{prop:form-identity}
For all $u,v\in H^1(C)$,
\begin{equation}\label{eq:form-identity}
 \Q_\Sigma(Ru,Rv)
 =\int_C\left(R^4u_\beta v_\beta+u_\alpha v_\alpha-uv\right)\dbda.
\end{equation}
\end{proposition}

\begin{proof}
First take $u,v$ smooth.  The product identity
\[
 \langle\nabla(Ru),\nabla(Rv)\rangle
 =R^2\langle\nabla u,\nabla v\rangle
   +\langle\nabla R,\nabla(Ruv)\rangle
\]
gives, after integration by parts,
\begin{align*}
 \Q_\Sigma(Ru,Rv)
 &=\int_\Sigma R^2\langle\nabla u,\nabla v\rangle\dA
   -\int_\Sigma Ruv\Delta R\dA\\
 &\quad+\int_{\partial\Sigma}Ruv\,\partial_\eta R\ds
       -\int_{\partial\Sigma}R^2uv\ds\\
 &=\int_\Sigma R^2
     \left(\langle\nabla u,\nabla v\rangle-Guv\right)\dA.
\end{align*}
Here we used \cref{eq:weighted-angle,eq:R-alpha-boundary}.
Substitution of the metric in \cref{eq:metric} gives
\cref{eq:form-identity} for smooth functions.

The identity extends to $H^1(C)$ by density and continuity, noting that $R$ and $1/R$ are smooth functions on the compact cylinder $C$. 

\end{proof}

Set
\begin{equation}\label{eq:a}
 a=R^4>0, \qquad \mathfrak q_a(u,v) =\int_C(a u_\beta v_\beta+u_\alpha v_\alpha-uv)\dbda,
\end{equation}
so that $\mathfrak{q}_a$ is the modified energy form in stream coordinates on the right-hand side of \cref{eq:form-identity}.

We will consider the self-adjoint operator associated with this form in $L^2(C,d\beta\,d\alpha)$. Thus $u\in D(\Aop)$ and $\Aop u=F$ mean that $u\in H^1(C)$, $F\in L^2(C)$, and
\[
\mathfrak q_a(u,v)=\int_C Fv\dbda \qquad\text{for all }v\in H^1(C).
\]
Integration by parts gives the differential expression
\begin{equation}\label{eq:A}
 \Aop u=-\partial_\beta(a\,\partial_\beta u) -\partial_{\alpha\alpha}u-u,
\end{equation}
with periodic conditions in $\alpha$ and Neumann conditions at $b_\pm$. Its domain is
\begin{equation}\label{eq:domain-A}
 D(\Aop)=\{u\in H^2(C):u_\beta=0\text{ on }\beta=b_\pm\}.
\end{equation}
Indeed, the boundary term is $\pm a u_\beta v$; unrestricted boundary traces of $v$ give $a u_\beta=0$, equivalently $u_\beta=0$. Smooth elliptic Neumann regularity gives the $H^2$ domain in \cref{eq:domain-A}, with the boundary conditions understood in the trace sense. Since $a$ has a positive lower bound, $\mathfrak q_a+2\|\cdot\|_{L^2}^2$ is coercive on $H^1(C)$. Compactness of $H^1(C)\hookrightarrow L^2(C)$ gives compact resolvent, and elliptic regularity makes all eigenfunctions smooth up to the boundary.

The operator $\mathcal{A}$ essentially corresponds to the Robin Laplacian in stream coordinates (and with the reweighting by $R$). Explicitly, we have:

\begin{lemma}\label[lemma]{lem:harmonic-quotient}
For a smooth function $u$ on $C$, we have
\begin{equation}\label{eq:exact-Laplacian-transform}
\Delta(Ru)=-RG\,\Aop u.
\end{equation}
In particular, 
\begin{equation}\label{eq:harmonic-transform}
 \Delta(Ru)=0\quad\Longleftrightarrow\quad \Aop u=0.
\end{equation}

On the boundary,
\begin{equation}\label{eq:Robin-transform}
 \partial_\eta(Ru)=\sigma Ru
 \quad\Longleftrightarrow\quad
 \partial_\eta u=(\sigma-1)u.
\end{equation}
In particular, the Steklov condition of eigenvalue $1$ is equivalent
to $u_\beta=0$ on $\partial C$.
\end{lemma}

\begin{proof}
Using $\Delta R=RG$,
\[
R\Delta(Ru)=\diver(R^2\nabla u)+R^2Gu.
\]
All gradients and divergences here use the induced metric, whereas
subscripts denote coordinate derivatives. With
$J=\sqrt{\det g}=(R^2G)^{-1}$, its divergence formula gives
\[
\begin{aligned}
\diver(R^2\nabla u) =J^{-1}\left[\partial_\beta(JR^2g^{\beta\beta}u_\beta)
  +\partial_\alpha(JR^2g^{\alpha\alpha}u_\alpha)\right] =R^2G\left(\partial_\beta(R^4u_\beta)+u_{\alpha\alpha}\right).
\end{aligned}
\]
Substituting the above implies \cref{eq:exact-Laplacian-transform}.

The boundary equivalence \eqref{eq:Robin-transform} follows from
$\partial_\eta(Ru)=Ru+R\partial_\eta u$.
Finally, \cref{eq:boundary-metric} identifies homogeneous Neumann
data with $u_\beta=0$ at the two ends.
\end{proof}

As the height $h$ was a 0-Robin eigenfunction for $\mathcal{Q}_\Sigma$, by the transformations above one should expect $t=h/R$ to give a 0-eigenfunction for $\mathfrak{q}_a$. Crucially, we observe that the derivative $t_\beta$ in stream coordinates is proportional to the support function $\xi=X\cdot N$ and hence does not change sign:

\begin{proposition} \label[proposition]{prop:monotone}
The smooth function
\begin{equation}\label{eq:t}
 t=\frac{h}{R}
\end{equation}
satisfies
\begin{equation}\label{eq:t-equation}
 \Aop t=0\quad\text{in }C,
 \qquad t_\beta=0\quad\text{on }\partial C,
 \qquad t_\beta>0\quad\text{in }C^\circ.
\end{equation}
More precisely,
\begin{equation}\label{eq:q-support}
t_\beta=\frac{\xi}{R^5G}.
\end{equation}
\end{proposition}

\begin{proof}
Recall that the coordinate function $h=X\cdot e$ satisfies $\Delta h=0$ and $\partial_\eta h=h$. Applying \cref{lem:harmonic-quotient} to $h=Rt$ gives the interior equation and the boundary condition in \cref{eq:t-equation}. 

Recall that $X=R e_r+he$ in the cylindrical parametrisation \cref{eq:cylindrical}, and that $e_\theta = e\times e_r = \partial_\alpha e_r$. It follows that the derivatives are
\[
 X_\beta=R_\beta e_r+h_\beta e, \qquad X_\alpha=R_\alpha e_r+R e_\theta+h_\alpha e.
\]
In the positively oriented orthonormal frame $(e_r,e_\theta,e)$,
\begin{align} \label{eq:triple-product}
 X\cdot(X_\beta\times X_\alpha) =R(hR_\beta-Rh_\beta) = -R^3t_\beta.
\end{align}
Since the orientation of $\Sigma$ is induced by the normal $N$ from \cref{lem:support}, we obtain $d\beta\wedge d\alpha =R^2\star d\alpha\wedge d\alpha =-R^2G\dA_g.$ Thus, the coordinate frame $(\partial_\beta,\partial_\alpha)$ is negatively oriented relative to $N$, and hence $X_\beta\times X_\alpha=-\sqrt{\det g}\,N.$ Taking the scalar product with $X$ and using \cref{eq:triple-product,eq:metric}, we obtain
\[
-R^3t_\beta =-\sqrt{\det g}\,\xi =-\frac{\xi}{R^2G}.
\]
Therefore,
\[
t_\beta=\frac{\xi}{R^5G},
\]
which proves \cref{eq:q-support}. Since $R,G>0$ on the closed cylinder, this expression is smooth up to the boundary. The positivity of $\xi$ in the interior and its vanishing on the boundary give the remaining assertions in \cref{eq:t-equation}.
\end{proof}

\section{A monotone-kernel spectral theorem on a cylinder}
\label{sec:spectral}

In this section, we will characterise the index and kernel of operators $\mathcal{A}$ that admit a monotone element of their kernel, as in \cref{prop:monotone}. In fact, we will formulate this as a somewhat general principle for operators on the cylinder $C=[b_-,b_+]\times(\R/2\pi\mathbb Z)$, which take to be equipped with the flat metric. In this section, we will fix $a$ to be an arbitrary smooth positive function on $C$ (that may depend on both coordinates). The main result of this section is \Cref{thm:monotone-criterion}.

Motivated by \cref{eq:A}, we define the operator $\Aop$ to be the Neumann realisation of 
\[
\Aop u=-\partial_\beta(a\,\partial_\beta u) -\partial_{\alpha\alpha}u-u.
\]
Define the Dirichlet operator
\begin{equation}\label{eq:B}
 \Bop w=-\partial_{\beta\beta}(aw)-w_{\alpha\alpha}-w,
 \qquad w\big|_{\partial C}=0.
\end{equation}
Note that its nondivergence form is
\begin{equation}\label{eq:B-expanded}
 \Bop w=-a w_{\beta\beta}-w_{\alpha\alpha}
        -2a_\beta w_\beta-(a_{\beta\beta}+1)w.
\end{equation}
This operator is uniformly elliptic, but need not be self-adjoint on $L^2(C)$. We write $n=\pm\partial_\beta$ for the outward unit normal on $C$. We use the strong maximum principle and Hopf's boundary point lemma in their local smooth-boundary forms \cite[Chapter~3]{GT}.

The operator $\mathcal{B}$ is essentially defined to be the operator which governs derivatives of the kernel of $\mathcal{A}$. Explicitly:

\begin{lemma}\label[lemma]{lem:intertwining}
For every smooth $u$,
\begin{equation}\label{eq:intertwining}
 \partial_\beta(\Aop u)=\Bop(u_\beta).
\end{equation}
If $u$ satisfies Neumann conditions at $b_\pm$, then $u_\beta$
satisfies Dirichlet conditions there.
\end{lemma}

\begin{proof}
Commuting the coordinate derivatives gives
\begin{align*}
 \partial_\beta(\Aop u) =-\partial_{\beta\beta}(a u_\beta)
   -\partial_{\alpha\alpha}(u_\beta)-u_\beta = \Bop(u_\beta).
\end{align*}
All derivatives of $a$ are retained inside
$\partial_{\beta\beta}(a u_\beta)$.
The boundary assertion is exactly $u_\beta|_{\partial C}=0$.
\end{proof}

We proceed by characterising the kernel of $\mathcal{B}$ when it admits a positive element. The proof follows a comparison principle against the positive solution:

\begin{lemma}\label[lemma]{lem:positive-Dirichlet}
Suppose $q\in C^\infty(C)$ satisfies
\begin{equation}\label{eq:positive-Dirichlet}
 \Bop q=0,\qquad q>0\text{ in }C^\circ,\qquad q=0\text{ on }\partial C.
\end{equation}
Then the following statements hold.
\begin{enumerate}
\item The outward flat-cylinder normal derivative satisfies
$\partial_nq<0$ at every boundary point.
\item There is no nonzero smooth real solution of
$\Bop w=\lambda w$, $w|_{\partial C}=0$, with $\lambda<0$.
\item The real Dirichlet kernel is $\ker\Bop=\Span\{q\}$.
\end{enumerate}
\end{lemma}

\begin{proof}
Choose $M_0\geq0$ such that
$M_0\geq\sup_C(a_{\beta\beta}+1)$.
Then $\Bop+M_0$ has nonnegative zeroth-order coefficient, and
\[
 (\Bop+M_0)q=M_0q\geq0.
\]
Apply Hopf's lemma to $-(\Bop+M_0)$, whose leading part is
positive and whose zeroth-order coefficient is nonpositive.
Since $q>0$ in the interior and vanishes on the boundary, it gives
$\partial_nq<0$.

In a boundary collar, take $\rho=\beta-b_-$ at the lower end or
$\rho=b_+-\beta$ at the upper end. Thus $\partial_n=-\partial_\rho$.
The fundamental theorem of calculus gives
\[
q(\rho,\alpha)=\rho q_1(\rho,\alpha),\qquad
q_1(\rho,\alpha)=\int_0^1
   (\partial_\rho q)(s\rho,\alpha)\,ds.
\]
After restricting to a smaller closed collar, $q_1$ is smooth and
$q_1(0,\alpha)=-\partial_nq(0,\alpha)>0$; it is also positive for
$\rho>0$ because $q>0$. The same construction gives
$w=\rho w_1$ for any smooth Dirichlet function $w$.
It follows that $w/q=w_1/q_1$ extends smoothly to the boundary,
and hence to all of $C$, with boundary value
\begin{equation}\label{eq:boundary-ratio}
 \left.\frac wq\right|_{\partial C}
 =\frac{\partial_nw}{\partial_nq}.
\end{equation}

Suppose next that $\Bop w=\lambda w$ with $\lambda<0$ and $w\not\equiv0$.  Replacing
$w$ by $-w$ if necessary, we can arrange $w$ to be positive somewhere in the
interior.  Set
\[
 m=\max_C\frac wq>0,
 \qquad z=mq-w\geq0.
\]
The maximum exists by \cref{eq:boundary-ratio}, and
\begin{equation}\label{eq:z-negative}
 (\Bop-\lambda)z=-\lambda m q>0\quad\text{in }C^\circ.
\end{equation}
An interior maximiser $p$ of $w/q$ would satisfy
$z(p)=0$, $Dz(p)=0$, and $D^2z(p)\geq0$. Evaluating
\cref{eq:B-expanded}, the first- and zeroth-order terms vanish:
\[
((\Bop-\lambda)z)(p)
=-a(p)z_{\beta\beta}(p)-z_{\alpha\alpha}(p)\leq0.
\]
This contradicts \cref{eq:z-negative}. The same observation
excludes every interior zero of $z$, so $z>0$ in $C^\circ$.
The maximum of $w/q$ is therefore attained at a boundary point $p$.
By \cref{eq:boundary-ratio},
\begin{equation}\label{eq:z-normal-zero}
 z(p)=0,
 \qquad \partial_nz(p)=m\partial_nq(p)-\partial_nw(p)=0.
\end{equation}
Choose $M\geq0$ with
$M\geq\sup_C(a_{\beta\beta}+1+\lambda)$.
The operator $\Bop-\lambda+M$ has nonnegative zeroth-order
coefficient, and
\[
 (\Bop-\lambda+M)z=-\lambda mq+Mz>0\quad\text{in }C^\circ.
\]
Since $z>0$ in the interior and $z(p)=0$, Hopf's lemma gives
$\partial_nz(p)<0$, contradicting \cref{eq:z-normal-zero}.
This proves the second assertion.

Finally, let $\Bop w=0$ with Dirichlet boundary data and set $m=\max_C(w/q)$,
$z=mq-w\geq0$. Then $\Bop z=0$ and $(\Bop+M_0)z=M_0z\geq0$. If $z\not\equiv0$, the
strong maximum principle gives $z>0$ in the interior, while the boundary point lemma
gives $\partial_nz<0$ at all boundary points. An interior maximiser of $w/q$ would
give an interior zero of $z$. A boundary maximiser would give $\partial_nz=0$ by
\cref{eq:boundary-ratio}. Both are impossible. Thus $z=0$ and $w=mq$.
\end{proof}

We can now characterise the nonpositive eigenspaces of $\mathcal{A}$ when it admits a monotone kernel element:

\begin{theorem}\label[theorem]{thm:monotone-criterion}
Let $a\in C^\infty(C)$ be strictly positive.  Suppose there is a smooth function $t$ satisfying
\begin{equation}\label{eq:abstract-t}
 \Aop t=0, \qquad t_\beta=0\text{ on }\partial C, \qquad t_\beta>0\text{ in }C^\circ.
\end{equation}
Then the operator $\Aop$ has exactly one negative eigenvalue, namely $-1$, with
eigenspace consisting of the constants.  Moreover,
\begin{equation}\label{eq:A-kernel}
 \ker\Aop=\Span\{t,\cos\alpha,\sin\alpha\}.
\end{equation}
Consequently, $\ind\mathfrak q_a=1$ and $\dim\ker\Aop=3$.
\end{theorem}

\begin{proof}
Set $q=t_\beta$. By \cref{lem:intertwining}, the function $q$
satisfies \cref{eq:positive-Dirichlet}, so \cref{lem:positive-Dirichlet} will apply. First let $u$ be a (real) Neumann eigenfunction of
$\Aop$ with eigenvalue $\lambda<0$.  Eigenfunctions are smooth, and differentiating
and recalling \cref{lem:intertwining} gives
\[
 \Bop(u_\beta)=\lambda u_\beta,
 \qquad u_\beta|_{\partial C}=0.
\]
By \cref{lem:positive-Dirichlet}, $u_\beta=0$.
Thus $u=u(\alpha)$ and
\[
 -u_{\alpha\alpha}-u=\lambda u.
\]
The eigenvalues of this operator on $\R/2\pi\mathbb Z$ are $n^2-1$, $n\in\mathbb Z$.
Its only negative eigenvalue is $-1$, with constant eigenfunctions. Conversely,
$\Aop1=-1$, so that eigenvalue occurs. Self-adjointness and compact resolvent then
imply $\ind\mathfrak q_a=1$.

Now suppose $u\in \ker\mathcal{A}$. Then $\Bop(u_\beta)=0$ (with Dirichlet data), and
the last part of \cref{lem:positive-Dirichlet} gives
$u_\beta=c_0t_\beta$ for a constant $c_0$.
Therefore
\[
 u=c_0t+v(\alpha),
 \qquad -v''-v=0,
\]
which proves \cref{eq:A-kernel}.
These three kernel functions are linearly independent: differentiating
a linear relation in $\beta$ first makes the coefficient of $t$ zero,
and $\cos\alpha,\sin\alpha$ are independent.
\end{proof}

\begin{corollary}[Sharp Poincar\'e inequality on the cylinder]
\label[corollary]{cor:cylinder-inequality}
Let $a$ and $t$ satisfy the hypotheses of \cref{thm:monotone-criterion} on
$C=[b_-,b_+]\times(\R/2\pi\mathbb Z).$ Then every real-valued $u\in H^1(C)$ with
zero mean with respect to the flat product measure, i.e. $\int_C u\dbda=0,$
satisfies
\begin{equation}\label{eq:cylinder-inequality}
\int_C u^2\dbda \leq \int_C\left(a|u_\beta|^2+|u_\alpha|^2\right)\dbda.
\end{equation}
No boundary condition on $u$ is required at $\beta=b_\pm$. Equality holds precisely
for $u\in\Span\{t,\cos\alpha,\sin\alpha\}$. In particular, the constant $1$ in
\cref{eq:cylinder-inequality} is optimal.
\end{corollary}

\begin{proof}
Let $\mathcal V=\{u\in H^1(C):\int_Cu\dbda=0\}$.
For $u\in\mathcal V$,
\[
\mathfrak q_a(1,u)=0,
\qquad \mathfrak q_a(1,1)=-|C|<0,
\qquad |C|=2\pi(b_+-b_-).
\]
If $\mathfrak q_a(u,u)<0$, the form would be negative definite
on $\Span\{1,u\}$, contrary to its index being one.
Thus $\mathfrak q_a(u,u)\geq0$ on $\mathcal V$, which proves
\cref{eq:cylinder-inequality}.

Suppose $u\in\mathcal V$ attains equality. For every
$v\in\mathcal V$ and $s\in\R$,
\[
0\leq\mathfrak q_a(u+sv,u+sv)
=2s\mathfrak q_a(u,v)+s^2\mathfrak q_a(v,v).
\]
Taking both signs of small $s$ gives $\mathfrak q_a(u,v)=0$.
Also $\mathfrak q_a(u,1)=0$. Decomposing any $H^1$ function
into its mean and a member of $\mathcal V$, we obtain
$\mathfrak q_a(u,v)=0$ for every $v\in H^1(C)$.
By the definition of the associated operator, $u\in D(\Aop)$
and $\Aop u=0$. The kernel is given by \cref{eq:A-kernel}.

Conversely, the kernel functions all have zero mean. For $t$,
integration using its Neumann data and angular periodicity gives
\[
\begin{aligned}
0=\int_C\Aop t\dbda
&=-\int_0^{2\pi}[a t_\beta]_{b_-}^{b_+}\,d\alpha
  -\int_{b_-}^{b_+}[t_\alpha]_0^{2\pi}\,d\beta
  -\int_Ct\dbda\\
&=-\int_Ct\dbda.
\end{aligned}
\]
The same conclusion for $\cos\alpha$ and $\sin\alpha$ is immediate.
Every kernel function therefore attains equality.
Taking $u=\cos\alpha\not\equiv0$ also proves optimality of the
constant.
\end{proof}

\section{The first Steklov eigenspace and uniqueness}
\label{sec:conclusion}

In this brief section we deduce the main \Cref{thm:main} using the stream coordinates and index characterisation above.

\begin{proof}[Proof of \cref{thm:steklov}]
Choose the flux axis and angular coordinate as in \cref{prop:angle}, and apply
\cref{prop:stream-cylinder} to construct the global cylinder. For $a=R^4$,
\Cref{prop:monotone} ensures that every hypothesis of
\Cref{thm:monotone-criterion} is satisfied. Consequently,
\[
 \ind\mathfrak q_a=1.
\]
Multiplication by $R$ and the coordinate diffeomorphism $\Phi=(\beta,\alpha):\Sigma\to C$ give an isomorphism
$H^1(C)\stackrel{\sim}{\to} H^1(\Sigma)$. By \cref{prop:form-identity}, this isomorphism preserves
the quadratic forms and therefore their indices:
\[
 \ind\Q_\Sigma=\ind\mathfrak q_a=1.
\]
\Cref{lem:index-criterion} now gives $\sigma_1=1$.

To identify the eigenspace, let $f\in\Eone$ and write $f=Ru$. By \cref{lem:harmonic-quotient},
\[
 \Aop u=0,\qquad u_\beta|_{\partial C}=0.
\]
Thus \cref{eq:A-kernel} yields
\[
 u=c_0\frac hR+c_1\cos\alpha+c_2\sin\alpha,
 \qquad
 f=c_0h+c_1R\cos\alpha+c_2R\sin\alpha.
\]
These are precisely linear combinations of the three ambient
coordinate functions in our orthonormal basis $\{e_1,e_2,e\}$.
Conversely, all three belong to $\Eone$ by
\cref{lem:coordinate-identities}.
Their linear independence follows from \cref{eq:A-kernel} and $R>0$.
Hence $\dim\Eone=3$, which gives
$\sigma_1=\sigma_2=\sigma_3=1<\sigma_4$.
\end{proof}

\begin{proof}[Proof of \cref{thm:main}]
By \cref{thm:steklov}, $\Sigma$ is a smooth free-boundary minimal annulus in the
unit ball, which is embedded by first (nonzero) Steklov eigenfunctions. The hypotheses of \Cref{thm:FS} are therefore satisfied, and it follows
that $\Sigma$ is congruent to $\Ccc$.
\end{proof}

\section{Classification of spherical Bernoulli solutions}
\label{sec:bernoulli}

Let $(\Omega,u)$ satisfy the hypotheses of
\cref{thm:bernoulli}. That is, $\Omega\subset\Sph^2$ is a smooth
annular domain and
\[
(\Delta_0+2)u=0,\qquad u>0
\quad\text{in }\Omega,
\qquad
u=0,\qquad |\nabla_0u|=1
\quad\text{on }\partial\Omega.
\]
Here $g_0$ is the round metric, $\nabla^0$ its Levi--Civita
connection, and $\nabla_0,\Delta_0$ the corresponding gradient
and Laplacian. We denote by $\nu$ the outward $g_0$-unit
conormal of $\partial\Omega$. Positivity of $u$ and the
boundary point lemma give
\[
\partial_\nu u=-1,\qquad \nabla_0u=-\nu
\quad\text{on }\partial\Omega.
\]
In this section, we exploit the well-known correspondence between solutions $u$ as above and \textit{branched} minimal immersions $\Sigma$. In the case of annular domains, we will show that $\Sigma$ is in fact an embedded free-boundary minimal annulus, allowing us to use our classification \cref{thm:main}. The (induced) metric and outward unit conormal on $\Sigma$ are denoted by $g$ and $\eta$, respectively; throughout this construction $g_0$ and $\nu$ refer to the spherical domain. We remark that the correspondence above may be regarded as the limiting case (with cap radius tending to 0) of the duality between free-boundary minimal surfaces in spherical caps and capillary minimal surfaces in the hemisphere, described in \cite{NZcaps}.

\subsection{Preliminaries on spherical Bernoulli solutions}
\label{sec:support}

In this subsection, we collect some preliminary results on the overdetermined problem in the sphere. These results are known to experts (see for instance \cite{EM} and references therein), but we include some proofs for the sake of readability.

\begin{lemma}[Homogeneous extension]
\label[lemma]{lem:homogeneous}
Define
\[
\mathcal C_\Omega:=\{rp:r>0,\ p\in\Omega\},
\qquad
U(rp):=ru(p).
\]
Then
\begin{equation}\label{eq:homogeneous-identities}
\Delta_{\R^3}U(rp)
=\frac1r(\Delta_0u+2u)(p),
\qquad
\nabla_{\R^3}U(rp)=u(p)p+\nabla_0u(p).
\end{equation}
Consequently, \cref{eq:OEP} is equivalent, under this
one-homogeneous extension, to
\begin{equation}\label{eq:Euclidean-Bernoulli}
\begin{cases}
U>0,\quad \Delta_{\R^3}U=0
    &\text{in }\mathcal C_\Omega,\\
U=0,\quad |\nabla U|=1
    &\text{on }\partial\mathcal C_\Omega\setminus\{0\},
\end{cases}
\end{equation}
where the boundary gradient is taken from within
$\mathcal C_\Omega$.
\end{lemma}

\begin{proof}
For a smooth function $V(r,p)$, the polar-coordinate
formulas are
\[
\Delta_{\R^3}V
=V_{rr}+\frac2rV_r+\frac1{r^2}\Delta_0V,
\qquad
\nabla_{\R^3}V=V_rp+\frac1r\nabla_0V.
\]
Applying them to $V(r,p)=ru(p)$ proves
\cref{eq:homogeneous-identities}. On the conical boundary,
$u=0$, so $|\nabla U|=|\nabla_0u|$. Conversely, a
one-homogeneous function is determined by its restriction
to $r=1$, and the same identities recover \cref{eq:OEP}.
\end{proof}

Define
\begin{equation}\label{eq:reconstruction}
X(p):=\nabla_0u(p)+u(p)p,
\qquad
\Bt:=\nabla_0^2u+ug_0,
\qquad
S:=\Bt^{\sharp_{g_0}}.
\end{equation}
Here $S_p:T_p\Sph^2\to T_p\Sph^2$ is characterised by
\[
g_0(S_pv,w)=\Bt_p(v,w).
\]
Equivalently,
\[
S(v)=\nabla^0_v\nabla_0u+uv.
\]
We first regard $X$ as a smooth map; immersion and injectivity
will be proved below.

\begin{lemma}\label[lemma]{lem:tensor}
The tensor $\Bt$ is symmetric and trace-free, and satisfies
the Codazzi identity
\[
(\nabla^0_v\Bt)(w,z)=(\nabla^0_w\Bt)(v,z).
\]
Moreover,
\begin{equation}\label{eq:DX}
dX_p(v)=S_pv,\qquad
X(p)\cdot p=u(p),\qquad
|X|^2=u^2+|\nabla_0u|^2.
\end{equation}
\end{lemma}

\begin{proof}
Let $D$ denote Euclidean differentiation. For
$v\in T_p\Sph^2$, the Gauss formula on the unit sphere gives
\[
D_v(\nabla_0u)
=\nabla^0_v\nabla_0u-v(u)p,
\qquad
D_v(up)=v(u)p+uv.
\]
Adding these identities gives $dX_p(v)=S_pv$.
The remaining identities follow from $\nabla_0u\perp p$.

Symmetry is immediate, and
\[
\tr_{g_0}\Bt=\Delta_0u+2u=0.
\]
For tangent vector fields $v,w$, commutation of the
derivatives of $X$ gives
\[
D_v(dX(w))-D_w(dX(v))-dX([v,w])=0.
\]
Taking the component tangent to $\Sph^2$ and using
$dX=S$, we obtain
\[
(\nabla^0_vS)w=(\nabla^0_wS)v.
\]
Pairing with $z$ and using $\nabla^0g_0=0$ gives the
stated Codazzi identity.
\end{proof}

\begin{lemma}[The $P$-function]
\label[lemma]{lem:P}
The function $P:=u^2+|\nabla_0u|^2$ satisfies
\begin{equation}\label{eq:P}
\Delta_0P=2|\Bt|_{g_0}^2,\qquad
P<1\ \text{in }\Omega,\qquad
P=1,\quad \partial_\nu P>0\ \text{on }\partial\Omega.
\end{equation}
\end{lemma}

\begin{proof}
Bochner's formula, $\operatorname{Ric}_{g_0}=g_0$, and
$\Delta_0u=-2u$ give
\[
\Delta_0|\nabla_0u|^2
=2|\nabla_0^2u|^2-2|\nabla_0u|^2,
\qquad
\Delta_0u^2=2|\nabla_0u|^2-4u^2.
\]
Also,
\[
|\Bt|^2
=|\nabla_0^2u|^2+2u\Delta_0u+2u^2
=|\nabla_0^2u|^2-2u^2.
\]
Adding the preceding identities proves the equation for
$P$. The boundary conditions give $P=1$ on $\partial\Omega$.

We claim that $\Bt\not\equiv0$. Otherwise, $dX=0$ by
\cref{lem:tensor}, so connectedness gives $X=a\in\R^3$ and
$u(p)=a\cdot p$. The boundary gradient condition excludes
$a=0$. Both boundary components would then be Jordan
curves contained in the great circle $\{a\cdot p=0\}$.
Each would have to equal that circle, contradicting their
disjointness.

Thus $P$ is not constant. Since $\Delta_0P\geq0$ and
$P=1$ on the boundary, the strong maximum principle gives
$P<1$ in $\Omega$, and the boundary point lemma gives
$\partial_\nu P>0$ on $\partial\Omega$; see
\cite[Chapter 3]{GT}.
\end{proof}

Let $K_\pm$ be the closed complementary discs:
\begin{equation}\label{eq:K}
\Sph^2=\overline\Omega\cup K_+\cup K_-,
\qquad
K_+\cap K_-=\varnothing,
\qquad
\partial\Omega=\partial K_+\sqcup\partial K_-.
\end{equation}
For a $g_0$-unit boundary tangent $T$, set
\begin{equation}\label{eq:kappa}
\kappa:=g_0(\nabla^0_TT,\nu).
\end{equation}
Since $\nu$ points out of $\Omega$ and into the
corresponding disc $K_\pm$, this is the geodesic curvature
towards the interior of that disc.

\begin{lemma}\label[lemma]{lem:boundary-B}
Along $\partial\Omega$,
\begin{equation}\label{eq:boundary-B}
\Bt(T,T)=\kappa,\qquad
\Bt(T,\nu)=0,\qquad
\Bt(\nu,\nu)=-\kappa,\qquad
\kappa>0.
\end{equation}
Consequently, $K_\pm$ are geodesically convex discs with
everywhere positive boundary geodesic curvature, and each
is contained in an open hemisphere.
\end{lemma}

\begin{proof}
Differentiating twice $u=0$ along the boundary gives
\[
0=T(Tu) =\nabla_0^2u(T,T) +g_0(\nabla_0u,\nabla^0_TT) = \nabla_0^2u(T,T)-\kappa.
\]
Likewise,
\[
0=T(\partial_\nu u) =\nabla_0^2u(T,\nu) + g_0(\nabla_0u,\nabla^0_T\nu) =\nabla_0^2u(T,\nu),
\]
because $\nabla_0u=-\nu$ and $\nabla^0_T\nu\perp\nu$. Since $u=0$ and $\tr_{g_0}\Bt=0$ on the boundary, these are the first three identities. Therefore
\[
\partial_\nu P =2\nabla_0^2u(\nu,\nabla_0u)+2u\partial_\nu u =-2\Bt(\nu,\nu)=2\kappa.
\]
The strict inequality in \cref{lem:P} proves $\kappa>0$.

The spherical convexity theorem \cite[Proposition 2.1]{BL} now implies that each $K_\pm$ is geodesically convex and lies on the inward side of every tangent great circle. Positive curvature excludes a great-circle segment in its boundary.

We also verify containment in an open hemisphere. Fix either disc $K$, parametrise its boundary by arclength $\ell$, and let $n=-\nu$ be its outward conormal. Writing $\gamma'(\ell)=T(\ell)$, differentiation of the orthonormal frame $(\gamma,T,n)$ gives
\[
n'=\kappa T.
\]
The supporting-hemisphere property says $n(\ell)\cdot y\leq0$ for every $y\in K$. Hence the vector
\[
v_K:=-\int_{\partial K}n\,\mathrm d\ell
\]
satisfies $v_K\cdot y\geq0$ on $K$. If equality held for some $y\in K$, the nonpositive continuous function $n(\ell)\cdot y$ would vanish identically. Differentiating and using $\kappa>0$ would give $T(\ell)\cdot y=0$ as well. Thus $y$ would be parallel to $\gamma(\ell)$ for every $\ell$, which is impossible for a regular boundary curve. Therefore
\[
v_K\cdot y>0\qquad(y\in K).
\]
In particular, $v_K\neq0$, and the hemisphere centred at $v_K/|v_K|$ contains $K$.
\end{proof}

\subsection{Annularity excludes branch points}
\label{sec:branches}

We first use a Hopf differential style argument to exclude the possibility of branch points of $X$:

\begin{proposition}\label[proposition]{prop:no-branches}
The endomorphism $S$ is invertible throughout $\overline\Omega$. Consequently, $X$ is a smooth conformal minimal immersion, with induced metric $g=X^*g_{\R^3}$ and unit normal along the immersion
\begin{equation}\label{eq:Gauss-normal}
N(p)=p.
\end{equation}
It satisfies
\begin{equation}\label{eq:free-boundary}
X(\Omega)\subset\B^3,\qquad X(\partial\Omega)\subset\Sph^2,\qquad dX(\eta)=X,
\end{equation}
where $\eta$ is the outward $g$-unit conormal.
\end{proposition}

\begin{proof}
Conformal uniformisation of a smooth annulus, with smooth extension to its boundary, gives a diffeomorphism from
\[
[a,b]\times(\R/2\pi\mathbb Z)
\]
onto $\overline\Omega$ for which the pulled-back round metric is
\[
g_0=e^{2\lambda}(ds^2+d\theta^2), \qquad \lambda\in C^\infty.
\]
We use the same notation for the pulled-back tensors.
Set
\[
f:=\Bt(\partial_s,\partial_s), \qquad h:=\Bt(\partial_s,\partial_\theta).
\]
As $\Bt$ is trace-free, we have $\Bt(\partial_\theta,\partial_\theta)=-f$. In these coordinates, the Codazzi equations reduce to
\[
h_s-f_\theta=0,\qquad f_s+h_\theta=0.
\]

It follows that $h_{ss}+h_{\theta\theta}=0$.
At $s=a,b$, the vector $\partial_s$ is conormal and
$\partial_\theta$ tangent to the boundary, so
$h=0$ there by \cref{eq:boundary-B}. The maximum principle
on the flat cylinder gives $h\equiv0$. The two first-order
equations then imply that $f$ is a real constant $m$.
Thus
\begin{equation}\label{eq:B-constant}
\Bt=m(ds^2-d\theta^2).
\end{equation}
At either boundary component,
\[
m=\Bt_{ss}
=e^{2\lambda}\Bt(\nu,\nu)
=-e^{2\lambda}\kappa<0.
\]
In particular, $m\neq0$. The matrices of $S$ and the
induced metric are therefore
\[
S=me^{-2\lambda}
\begin{pmatrix}1&0\\0&-1\end{pmatrix},
\qquad
g=m^2e^{-2\lambda}(ds^2+d\theta^2).
\]
This proves invertibility, conformality, and regularity
up to the boundary.

Since $dX_p(T_p\Sph^2)=p^\perp$, the vector $N(p)=p$ is
a unit normal. The shape operator $\mathcal S$ satisfies
\[
\mathcal S(dX_p(v))=-dN_p(v)=-v.
\]
Under the identification of the image tangent plane with
$p^\perp$, it is therefore $-S_p^{-1}$. The eigenvalues
of $S_p$ are opposite and nonzero, so
$\tr\mathcal S=-\tr(S_p^{-1})=0$. Hence $X$ is minimal.

The ball and sphere inclusions follow from
$|X|^2=P$ and \cref{lem:P}. Along the boundary,
\[
X=-\nu,\qquad
dX(\nu)=S\nu=-\kappa\nu=\kappa X,
\qquad
dX(T)=\kappa T.
\]
Consequently, $g(\nu,\nu)=\kappa^2$ and $g(\nu,T)=0$.
The outward $g$-unit conormal is thus
$\eta=\nu/\kappa$, and $dX(\eta)=X$.
\end{proof}

\subsection{Spherical polarity and global embeddedness}
\label{sec:embedding}

Until injectivity has been established, we regard
$(\overline\Omega,g)$ as the parameter annulus of the
immersion $X$.

\begin{lemma}[Polarity of a positively curved spherical disc]
\label[lemma]{lem:polars}
Let $K\subset\Sph^2$ be a closed geodesically convex disc,
contained in an open hemisphere, whose smooth boundary
has strictly positive inward geodesic curvature $\kappa$.
Let $n$ be its outward unit conormal. Define the
negative polar by
\begin{equation}\label{eq:polar-definition}
D:=\{y\in\Sph^2:y\cdot p\leq0
                \text{ for every }p\in K\}.
\end{equation}
Then $D$ is a geodesically convex closed disc in an open hemisphere, with smooth boundary, and
\begin{equation}\label{eq:polar-map}
\partial K\rightarrow\partial D,
\qquad p\mapsto n(p),
\end{equation}
is a diffeomorphism. At $n(p)$, the outward conormal of $D$ is $p$, and its inward geodesic curvature is $1/\kappa(p)$.
\end{lemma}

\begin{proof}
Choose $h\in\Sph^2$ with $h\cdot p>0$ on $K$.
Compactness makes this inequality uniform, so $-h$ lies
in the interior of $D$. Next fix $p_0\in\Int K$.
Every $y\in D$ satisfies
\[
y\cdot p_0<0.
\]
Indeed, equality would make $y$ tangent to $\Sph^2$ at
$p_0$, and the points
$\cos t\,p_0+\sin t\,y$ would belong to $K$ for sufficiently
small $t>0$, contradicting the defining inequality for $D$.
Thus $D$ lies in the open hemisphere centred at $-p_0$.

The set
\[
\mathcal L:=\{v\in\R^3:v\cdot p\leq0
                      \text{ for every }p\in K\}
\]
is a closed convex cone and $D=\mathcal L\cap\Sph^2$.
The map
\[
y\longmapsto\frac{y}{-y\cdot p_0}
\]
identifies $D$ with the compact convex planar set
$\mathcal L\cap\{v:-v\cdot p_0=1\}$, which has nonempty
relative interior. Hence $D$ is a closed topological
disc; convexity of $\mathcal L$ gives geodesic convexity
of $D$.

At $p\in\partial K$, the tangent great circle supports
$K$, so
\[
n(p)\cdot x\leq0\quad(x\in K),
\qquad n(p)\cdot p=0.
\]
Therefore $n(p)\in\partial D$.
Conversely, if $y\in\partial D$, compactness of $K$
implies that $y\cdot p=0$ at some $p\in K$; otherwise
all the defining inequalities would remain strict near
$y$, making it interior to $D$. Such a contact point
belongs to $\partial K$. It is unique, because two
distinct contact points would force their shorter
great-circle segment to lie in $\partial K$, contrary
to $\kappa>0$. Smoothness of $\partial K$ then implies
$y=n(p)$. This proves bijectivity.

Parametrise $\partial K$ by arclength $s$ and write
$T=p_s$. Its Euclidean frame equations are
\[
p_s=T,\qquad
T_s=-p-\kappa n,\qquad
n_s=\kappa T.
\]
Since $\kappa>0$, the conormal map is an immersion.
Its bijectivity and compactness therefore make it a
diffeomorphism onto a smooth boundary.
The arclength $s_D$ of this image satisfies
$ds_D=\kappa\,ds$, and its unit tangent is $T$. Hence
\[
\frac{dT}{ds_D}=-n-\frac1\kappa p.
\]
The inequalities defining $D$ show that $p$ is its
outward conormal at $n(p)$. The last identity therefore
gives inward geodesic curvature $1/\kappa$.
\end{proof}

Apply \cref{lem:polars} to $K_+$ and $K_-$, and denote
their negative polars by $D_+$ and $D_-$.
Their disjointness has not yet been established.
Since the outward conormal of $K_\pm$ is $-\nu$,
the boundary restriction of $X$ is precisely the
conormal map:
\begin{equation}\label{eq:boundary-polar}
X(p)=-\nu(p)=n_{K_\pm}(p),
\qquad p\in\partial K_\pm.
\end{equation}

\begin{lemma}[Radial projection near the boundary]
\label[lemma]{lem:collar}
The map
\begin{equation}\label{eq:radial-map}
F:=\frac{X}{|X|}:\overline\Omega\longrightarrow\Sph^2
\end{equation}
is well defined and is a local diffeomorphism in
$\Omega$. Near each boundary component, it maps a
neighbourhood on the $\Omega$ side homeomorphically
onto a neighbourhood on the exterior side of the
corresponding polar disc $D_\pm$.

More precisely, let $r$ be inward spherical distance
from $\partial\Omega$, and let $\rho_D$ be signed
spherical distance from $\partial D_\pm$, positive
outside $D_\pm$. Then
\begin{equation}\label{eq:collar-expansion}
\rho_D(F(p,r))
=\frac{\kappa(p)}2r^2+O(r^3),
\qquad p\in\partial K_\pm.
\end{equation}
\end{lemma}

\begin{proof}
In the interior, $X\cdot p=u>0$, while $|X|=1$ on the
boundary. Hence $X$ never vanishes and $F$ is defined
on $\overline\Omega$. For a positively oriented
$g_0$-orthonormal frame $(v_1,v_2,p)$,
\[
X\cdot(dX(v_1)\times dX(v_2))=u\det S.
\]
Thus the oriented Jacobian of $F$, with respect to
the round area forms, is
\begin{equation}\label{eq:radial-Jacobian}
J_F=\frac{u\det S}{|X|^3}\neq0
\quad\text{in }\Omega.
\end{equation}

We next determine which side of $\partial D_\pm$
contains the image near the boundary. For
$p\in\partial\Omega$, put $y_0=X(p)=-\nu(p)$ and use
the inward normal parametrisation
\[
\Psi(p,r):=\cos r\,p+\sin r\,y_0,
\qquad 0\leq r<\delta.
\]
For sufficiently small $\delta>0$, this parametrises
a neighbourhood of this boundary component in $\overline\Omega$.
We write $X(p,r)$ and $F(p,r)$ for composition with
$\Psi$. The boundary identities give
\begin{equation}\label{eq:u-collar}
u(\Psi(p,r))
=r-\frac{\kappa(p)}2r^2+O(r^3),
\qquad
X_r(p,0)=-\kappa(p)y_0.
\end{equation}
Writing $L=|X|$, we obtain
\[
L(p,0)=1,\qquad L_r(p,0)=-\kappa(p),
\qquad F_r(p,0)=0.
\]

Since $dX_q(v)\perp q$, we have
$X_r(p,r)\cdot\Psi(p,r)=0$. Differentiating at $r=0$
yields
\[
X_{rr}(p,0)\cdot p
=-X_r(p,0)\cdot y_0=\kappa(p).
\]
Both $X(p,0)$ and $X_r(p,0)$ are orthogonal to $p$,
so differentiating $F=X/L$ twice also gives
\[
F_{rr}(p,0)\cdot p=\kappa(p).
\]
By \cref{lem:polars},
$\nabla_0\rho_D(y_0)=p$. Since $F_r(p,0)=0$, the chain
rule therefore gives
\[
\left.\frac{d}{dr}\rho_D(F(p,r))\right|_{r=0}=0,
\qquad
\left.\frac{d^2}{dr^2}\rho_D(F(p,r))\right|_{r=0}
=\kappa(p).
\]
Taylor's formula proves \cref{eq:collar-expansion},
uniformly along the compact boundary component.

To obtain the asserted homeomorphism, let $s$ be
arclength on $\partial\Omega$, and use target
coordinates $(\sigma,\rho_D)$, where $\sigma$ is
arclength on $\partial D_\pm$ extended along its
normal geodesics. At $r=0$,
\[
F_s(s,0)=\kappa(s)T(s),
\]
and the boundary map is a diffeomorphism by
\cref{lem:polars}. After shrinking the neighbourhood,
we may therefore replace the source coordinate $s$
by the target coordinate $\sigma$.
In these coordinates, Taylor's formula gives
\[
F(\sigma,r)=\bigl(\sigma,r^2A(\sigma,r)\bigr),
\qquad A\in C^\infty,\quad A(\sigma,0)>0.
\]
The change of variable
\[
\tau=r\sqrt{A(\sigma,r)}
\]
is a smooth diffeomorphism near $r=0$ on the source
side. The map $F$ consequently becomes
\[
(\sigma,\tau)\longmapsto(\sigma,\tau^2),
\qquad \tau\geq0.
\]
This is a homeomorphism onto the exterior side
$\rho_D\geq0$. Compactness allows the neighbourhoods
to be chosen around the entire boundary component.
\end{proof}

\begin{proposition}[Global injectivity]
\label[proposition]{prop:embedding}
The polar discs $D_+$ and $D_-$ are disjoint, and
\begin{equation}\label{eq:radial-homeomorphism}
F:\overline\Omega\longrightarrow
\Sph^2\setminus(\Int D_+\cup\Int D_-)
\end{equation}
is a homeomorphism whose restriction to $\Omega$
is a diffeomorphism onto
$\Sph^2\setminus(D_+\cup D_-)$.
Moreover, $X$ is a smooth embedding with image a
properly embedded free-boundary minimal annulus in
$\overline{\B^3}$.
\end{proposition}

\begin{proof}
The boundary diffeomorphism
\[
F|_{\partial K_\pm}:\partial K_\pm\to\partial D_\pm
\]
extends to a homeomorphism
$\Psi_\pm:K_\pm\to D_\pm$. Indeed, after choosing
homeomorphisms of both discs with the closed unit
disc, a boundary homeomorphism
$h:S^1\to S^1$ extends by
\[
re^{i\theta}\longmapsto r\,h(e^{i\theta}),
\qquad 0\leq r\leq1.
\]
Define
\[
\widehat F:\Sph^2\longrightarrow\Sph^2
\]
to equal $F$ on $\overline\Omega$ and $\Psi_\pm$ on
$K_\pm$. These definitions agree on their common
boundaries, so $\widehat F$ is continuous.

It is a local homeomorphism in $\Omega$ by
\cref{eq:radial-Jacobian}, and in $\Int K_\pm$ by
construction. Near $\partial K_\pm$, the map
$\Psi_\pm$ takes the $K_\pm$ side to the interior
side of $D_\pm$, while \cref{lem:collar} takes the
$\Omega$ side to the exterior side. Their inverses
agree on $\partial D_\pm$, so they combine into a
continuous local inverse across that boundary.
Thus $\widehat F$ is a local homeomorphism everywhere.

Its image is open by local invertibility and closed
by compactness, hence is all of $\Sph^2$.
A proper local homeomorphism is a covering map.
Since its domain is connected and its target
$\Sph^2$ is simply connected, this covering has one
sheet. Therefore $\widehat F$ is a homeomorphism.
In particular, the images $D_\pm=\widehat F(K_\pm)$
are disjoint, and restriction to $\overline\Omega$
gives \cref{eq:radial-homeomorphism}.
Its interior inverse is smooth by
\cref{eq:radial-Jacobian}.

Finally, $X(p)=X(q)$ implies $F(p)=F(q)$, hence $p=q$.
Thus $X$ is an injective immersion of a compact
manifold with boundary, and therefore an embedding.
The free-boundary and properness assertions follow
from \cref{eq:free-boundary}.
\end{proof}

\begin{remark}\label[remark]{rem:two-spheres}
The original domain $\Omega$ is the Gauss image of the dual
surface: writing $N$ now as a field on the image, $N(X(p))=p$.
Its radial image is instead
$\Sph^2\setminus(D_+\cup D_-)$.
The Gauss image and radial image are distinct constructions and
have not been assumed to coincide. Moreover, the boundary behaviour of
$F$ is modelled by $\tau\mapsto\tau^2$, so
\cref{eq:radial-homeomorphism} is not a
diffeomorphism at the boundary. This does not affect
the smoothness of $X$ or the boundary regularity of
the stream coordinates in \cref{prop:stream-cylinder}.
\end{remark}

\subsection{Classification via the dual annulus}
\label{sec:classification}

Recall the constants from \cref{eq:constants}:
\[
T\tanh T=1,\qquad
s_*=\tanh T,\qquad
c=\frac1{T\cosh T}
=s_*\sqrt{1-s_*^2}.
\]

\begin{corollary}\label[corollary]{cor:catenoid}
Up to an orthogonal transformation of $\R^3$, the
image of the dual annulus is parametrised by
\begin{equation}\label{eq:catenoid}
X_*(v,\alpha)
=c\bigl(\cosh v\,e_r(\alpha)-ve\bigr),
\qquad
(v,\alpha)\in[-T,T]\times(\R/2\pi\mathbb Z),
\end{equation}
where $(e_1,e_2,e)$ is an orthonormal basis and
$e_r(\alpha)=\cos\alpha\,e_1+\sin\alpha\,e_2$.
\end{corollary}

\begin{proof}
By \cref{prop:embedding}, the dual surface
satisfies the hypotheses of \cref{thm:main}, and hence
is the critical catenoid. The minus sign in the axial
component of \cref{eq:catenoid} is an orthogonal
reflection of the usual parametrisation.
\end{proof}

\begin{proof}[Proof of \cref{thm:bernoulli}]
The dual map commutes with orthogonal transformations:
if $\widetilde\Omega=O\Omega$ and
$\widetilde u=u\circ O^{-1}$, then
$\widetilde X(Op)=OX(p)$. Thus, by \cref{cor:catenoid}, we may
apply the same orthogonal transformation to $(\Omega,u)$ and its
dual and assume that
$X(\overline\Omega)$ is the surface in \cref{eq:catenoid}.
Its unit normal with positive support is
\begin{equation}\label{eq:model-normal}
N_*(v,\alpha)
=\sech v\,e_r(\alpha)+\tanh v\,e.
\end{equation}
Indeed, this vector has unit length, is orthogonal
to $X_{*,v}$ and $X_{*,\alpha}$, and satisfies
\begin{equation}\label{eq:model-support}
X_*\cdot N_*
=c(1-v\tanh v)>0
\qquad(|v|<T).
\end{equation}
The original normal satisfies $N(X(p))=p$ and
$X(p)\cdot p=u(p)>0$, so it agrees with $N_*$,
rather than $-N_*$.

The Gauss image of the interior is therefore
\[
\Omega=\{p\in\Sph^2:|p\cdot e|<\tanh T\}.
\]
Writing $s=p\cdot e=\tanh v$, so that
$v=\artanh s$, the identity $u=X\cdot N$ becomes
\[
u(p)=c\bigl(1-s\,\artanh s\bigr).
\]
This proves
\cref{eq:classified-domain,eq:classified-function}.
The converse is verified in \cref{prop:explicit}.
\end{proof}

\begin{corollary}[Classification of the annular homogeneous solution]
\label[corollary]{cor:cone}
Let $U$ be a positive one-homogeneous harmonic function
on an open cone $\mathcal C\subset\R^3\setminus\{0\}$
whose spherical link is a smooth annulus.
Assume that $U$ is smooth up to
$\partial\mathcal C\setminus\{0\}$ and satisfies
$U=0$ and $|\nabla U|=1$ there.
Then there is $e\in\Sph^2$ such that, writing
$r=|x|$, $z=x\cdot e$, and $\rho=|x-ze|$,
\begin{equation}\label{eq:cone-domain}
\mathcal C
=\{x\neq0:|z|<s_*r\}
=\{x\neq0:|z|<\sinh T\,\rho\},
\end{equation}
and
\begin{equation}\label{eq:cone-solution}
U(x)
=cr\left(1-\frac zr\artanh\frac zr\right)
\qquad(x\in\mathcal C).
\end{equation}
Its extension by zero outside $\mathcal C$ and at the origin is
locally Lipschitz and is the axially symmetric annular Alt--Caffarelli
solution. No stability or minimising property is asserted.
\end{corollary}

\begin{proof}
Restrict $U$ to $\mathcal C\cap\Sph^2$ and apply
\cref{lem:homogeneous,thm:bernoulli}.
Homogeneity gives \cref{eq:cone-solution}, while
$r^2=\rho^2+z^2$ gives
\[
|z|<s_*r
\quad\Longleftrightarrow\quad
|z|<
\frac{s_*}{\sqrt{1-s_*^2}}\rho
=\sinh T\,\rho.
\]

The zero extension is continuous across the conical
boundary because $u=0$ there, and at the origin
because $|U(x)|\leq C|x|$.
By \cref{eq:homogeneous-identities,lem:P},
$|\nabla U|\leq1$ in $\mathcal C$.
Integration by parts across the smooth conical
boundary introduces no boundary term in the first
weak derivatives, since $U$ vanishes there.
To include the origin, perform the same calculation
outside $B_\varepsilon$. The additional boundary
term on $\partial B_\varepsilon$ is
$O(\varepsilon^3)$, using $|U|=O(\varepsilon)$
and $|\partial B_\varepsilon|=O(\varepsilon^2)$.
It therefore tends to zero.

Thus the weak gradient of the zero extension is
$\nabla U$ in $\mathcal C$ and zero outside, with
norm at most one. The extension belongs to
$W^{1,\infty}_{\mathrm{loc}}(\R^3)$ and is locally
Lipschitz. The harmonic and free-boundary equations
are those of
\cref{lem:homogeneous,prop:explicit}.
This is the classical axially symmetric annular homogeneous example;
see \cite{AC} and \cite[Section 1]{HKM}.
\end{proof}

\section{Free-boundary minimal annuli in spherical caps}
\label{sec:spherical-caps}

In this section, we describe how our methods may be used to obtain rotational symmetry for embedded free-boundary minimal annuli in spherical caps. Aside from testing the method, our main motivation for including this section was to cover the case of free-boundary minimal annuli in the hemisphere, which correspond to minimal tori in $\mathbb{S}^3$, and so compare our methods to the closed setting. We expect that similar methods should also establish rotational symmetry of free-boundary minimal annuli in geodesic balls in the hyperbolic space form $\mathbb{H}^3$. 

Fix $o=(1,0,0,0)\in\Sph^3$ and $0<\vartheta<\pi/2$, and consider the spherical cap
\[
B_\vartheta(o)
=\{X\in\Sph^3:X\cdot o>\cos\vartheta\}.
\]
Let $X:\Sigma\hookrightarrow\overline{B_\vartheta(o)}$ be a smooth properly embedded free-boundary minimal annulus. Write
\[
X=(X_0,Y),\qquad Y=(X_1,X_2,X_3)\in\R^3,
\]
and let $g$ and $\eta$ denote the induced metric and outward unit conormal. Minimality and the free-boundary condition give
\begin{equation}\label{eq:cap-coordinate-equations}
(\Delta+2)X=0,\qquad \partial_\eta X =\cot\vartheta\,X-\csc\vartheta\,o.
\end{equation}
In particular, along $\partial\Sigma$,
\begin{equation}\label{eq:cap-boundary-coordinates}
X_0=\cos\vartheta,\qquad \partial_\eta X_0=-\sin\vartheta,\qquad \partial_\eta Y=\cot\vartheta\,Y.
\end{equation}
Define the modified energy form on $H^1(\Sigma)$
\begin{equation}\label{eq:cap-form}
\Q_\vartheta(f,h) = \int_\Sigma \bigl(\langle\nabla f,\nabla h\rangle-2fh\bigr)\dA -\cot\vartheta\int_{\partial\Sigma}fh\ds.
\end{equation}

\begin{lemma}[The spherical angular coordinate]
\label[lemma]{lem:cap-angle}
There is a unit vector $e\in\R^3$ such that, setting
\[
H=Y\cdot e,\qquad R=|Y-He|, \qquad Y=Re_r(\alpha)+He,
\]
one has $R>0$ and $d\alpha\neq0$ on $\Sigma$. Both boundary restrictions of $\alpha$ are diffeomorphisms onto $\R/2\pi\mathbb Z$. Moreover, the spherical unit normal $N=(N_0,N')$ can be chosen so that
\begin{equation}\label{eq:cap-support}
\chi:=-N_0>0\quad\text{in }\Sigma^\circ, \qquad \chi=0\quad\text{on }\partial\Sigma.
\end{equation}
\end{lemma}

\begin{proof}
The two-piece and radial-transversality results of Naff--Zhu \cite[Corollary~3.6 and Proposition~3.12]{NZcaps} give $o\notin X(\Sigma)$ and a normal for which
\[
\langle N,\partial_\rho\rangle>0 \quad\text{in }\Sigma^\circ,
\]
where $\rho$ is distance from $o$. Since
\[
X=(\cos\rho,\sin\rho\,n),\qquad n:=\frac{Y}{|Y|}, \qquad \langle N,\partial_\rho\rangle=-\frac{N_0}{\sin\rho},
\]
this proves \cref{eq:cap-support}. The boundary vanishing also follows directly from $\langle N,X\rangle=\langle N,\partial_\eta X\rangle=0$.

We first identify the radial image. Use the convention
\[
A(V,W)=\langle D_VN,dX(W)\rangle.
\]
Then $N$ satisfies $\Delta N+|A|^2N=0$, hence $\Delta\chi+|A|^2\chi=0$. Hopf's lemma gives $\partial_\eta\chi<0$. Since the boundary sphere is umbilic, $A(T,\eta)=0$ for a unit boundary tangent $T$, and
\[
\partial_\eta\chi =\sin\vartheta\,A(\eta,\eta).
\]
Thus, $A(\eta,\eta)<0$ on both boundary components. 

On $\partial\Sigma$, the curve $n=Y/\sin\vartheta$ has, with respect to its round arclength $\ell$,
\[
n_{\ell\ell}=-n-\kappa N', \qquad \kappa=-\sin\vartheta\,A(\eta,\eta)>0.
\]
Each boundary curve therefore bounds a geodesically convex disc $D_\pm\subset\Sph^2$, contained in an open hemisphere, whose outward conormal is $N'$; see \cite[Proposition~2.1]{BL} and the hemisphere argument in \cref{lem:boundary-B}.

The radial map $n$ is a local diffeomorphism in the interior: its Jacobian is
\[
J_n=\frac{\chi}{|Y|^3}>0.
\]
It maps the interior side of each boundary to the exterior side of $D_\pm$. Indeed, let $s\geq0$ be inward $g$-distance from the boundary and let $\delta_D$ be signed round distance from $\partial D_\pm$, positive outside. At $s=0$,
\[
n_s=0,\qquad n_{ss}\cdot N' =-\frac{A(\eta,\eta)}{\sin\vartheta} =\frac{\kappa}{\sin^2\vartheta},
\]
so
\[
\delta_D(n(p,s)) =\frac{\kappa(p)}{2\sin^2\vartheta}s^2+O(s^3).
\]
As in \cref{lem:collar}, the nonzero tangential derivative and this expansion give a homeomorphism between boundary neighbourhoods on the indicated sides.

Fill the two boundary circles of $\Sigma$ by discs and extend their boundary maps homeomorphically onto $D_\pm$. The resulting map $\Sph^2\to\Sph^2$ is a local homeomorphism, including across the two boundary circles. It is therefore a one-sheeted covering. Consequently, $D_+\cap D_-=\varnothing$ and
\[
n:\Sigma \rightarrow \Sph^2\setminus(\Int D_+\cup\Int D_-)
\]
is a homeomorphism, with interior a diffeomorphism.

Green's identity and \cref{eq:cap-coordinate-equations,eq:cap-boundary-coordinates} give, componentwise,
\[
\begin{aligned}
0 =\int_\Sigma(X_0\Delta Y-Y\Delta X_0)\dA = \int_{\partial\Sigma}(X_0\partial_\eta Y-Y\partial_\eta X_0)\ds =\csc\vartheta\int_{\partial\Sigma}Y\ds.
\end{aligned}
\]
Since $Y=\sin\vartheta\,n$ and $ds=\sin\vartheta\,d\ell$ on the boundary,
\[
\int_{\partial D_+}n\,d\ell+ \int_{\partial D_-}n\,d\ell=0.
\]
The convex-cap lemma \cref{lem:cap-flux} therefore supplies $e\in\Int D_+$ with $-e\in\Int D_-$. The radial image avoids both poles, giving $R>0$, and the same lemma gives the boundary angular diffeomorphisms.

Finally,
\[
d\alpha(v) =\frac{\langle(0,e_\theta),dX(v)\rangle}{R}.
\]
If this covector vanished at an interior point, then $(0,e_\theta)$, which is tangent to $\Sph^3$, would be parallel to $N$. This would force $N_0=0$, contradicting \cref{eq:cap-support}. Boundary nonvanishing follows from the boundary angular diffeomorphisms.
\end{proof}

\begin{theorem}[Spherical-cap spectral rigidity]
\label[theorem]{thm:cap-rigidity}
For every smooth properly embedded free-boundary minimal annulus in $\overline{B_\vartheta(o)}$, with $0<\vartheta<\pi/2$,
\[
\ind\Q_\vartheta=1,\qquad \ker\Q_\vartheta=\Span\{X_1,X_2,X_3\}.
\]
The eigenvalues of
\[
(\Delta+2)f=0\quad\text{in }\Sigma^\circ, \qquad \partial_\eta f=\tau f\quad\text{on }\partial\Sigma,
\]
indexed with multiplicity, satisfy
\[
\tau_0=-\tan\vartheta < \tau_1=\tau_2=\tau_3=\cot\vartheta < \tau_4.
\]
In particular, $\Sigma$ is rotationally symmetric.
\end{theorem}

\begin{proof}
Choose the axis from \cref{lem:cap-angle}. Expanding $(\Delta+2)(Re^{i\alpha})=0$ and the common boundary condition of $X_1,X_2$ gives
\begin{equation}\label{eq:cap-angular-equations}
\Delta R=(G-2)R,\qquad \diver(R^2\nabla\alpha)=0,\qquad G:=|\nabla\alpha|^2>0,
\end{equation}
and
\begin{equation}\label{eq:cap-angular-boundary}
\partial_\eta R=\cot\vartheta\,R,\qquad \partial_\eta\alpha=0.
\end{equation}
The proof of \cref{prop:stream-cylinder} now applies:
\[
d\beta= R^2\star d\alpha,\qquad (\beta,\alpha):\Sigma \rightarrow C=[b_-,b_+]\times(\R/2\pi\mathbb Z)
\]
is a diffeomorphism up to the boundary, and
\[
g=\frac{d\beta^2}{R^4G}+\frac{d\alpha^2}{G}, \qquad dA=\frac{1}{R^2G}\,d\beta\,d\alpha.
\]

For smooth $u,v$, integration by parts gives
\begin{align}
\Q_\vartheta(Ru,Rv)
&=\int_\Sigma \left(R^2\langle\nabla u,\nabla v\rangle -Ruv\Delta R-2R^2uv\right)\dA \notag\\
&=\int_\Sigma R^2\left(\langle\nabla u,\nabla v\rangle-Guv\right)\dA
 \notag\\
&=\int_C \left(R^4u_\beta v_\beta+u_\alpha v_\alpha-uv\right)\dbda.
\label{eq:cap-form-transform}
\end{align}
The boundary terms cancel by \cref{eq:cap-angular-boundary}; the interior curvature terms cancel by \cref{eq:cap-angular-equations}. As in \cref{prop:form-identity}, the identity extends to $H^1(C)$.

Set
\[
\Aop=-\partial_\beta(R^4\partial_\beta) -\partial_{\alpha\alpha}-1
\]
with Neumann conditions at $b_\pm$. The corresponding pointwise and boundary identities are
\[
(\Delta+2)(Ru)=-RG\,\Aop u, \qquad \partial_\eta(Ru)-\cot\vartheta\,Ru =R\partial_\eta u.
\]
Since $h$ satisfies the same equation and Robin condition as $X_1,X_2$, the quotient $t=h/R$ satisfies
\[
\Aop t=0,\qquad t_\beta=0\quad\text{on }\partial C.
\]

To determine the sign of $t_\beta$, observe that
\[
Y\cdot(Y_\beta\times Y_\alpha)=-R^3t_\beta.
\]
The four-dimensional determinant of $X,X_\beta,X_\alpha,o$ gives
\[
\left|Y\cdot(Y_\beta\times Y_\alpha)\right| =\chi\sqrt{\det g}.
\]
After reversing $\beta$ if necessary,
\[
t_\beta=\frac{\chi}{R^5G}>0 \quad\text{in }C^\circ, \qquad t_\beta=0\quad\text{on }\partial C.
\]
Thus \cref{thm:monotone-criterion} applies with $a=R^4$. Together with \cref{eq:cap-form-transform}, it gives
\[
\ind\Q_\vartheta=1,\qquad \ker\Q_\vartheta = R\,\Span\{t,\cos\alpha,\sin\alpha\} =\Span\{X_1,X_2,X_3\}.
\]

It remains to identify the spectral parameter. As in \cite{LMcaps}, the positive function $X_0\geq\cos\vartheta>0$ satisfies $(\Delta+2)X_0=0$ and $\partial_\eta X_0=-\tan\vartheta\,X_0$. The ground-state identity is
\begin{equation}\label{eq:cap-ground-state}
\int_\Sigma\left(|\nabla f|^2-2f^2\right)\dA +\tan\vartheta\int_{\partial\Sigma}f^2\ds = \int_\Sigma X_0^2
\left|\nabla\left(\frac{f}{X_0}\right)\right|^2\dA.
\end{equation}
For zero-trace functions this also gives $\lambda_1^D(-\Delta)>2$, so the frequency-$2$ Dirichlet-to-Neumann operator is well defined. Equation~\eqref{eq:cap-ground-state} shows that its lowest eigenvalue is $-\tan\vartheta$, simple with eigenfunction $X_0$. 

The index of $\Q_\vartheta$ counts the eigenvalues strictly below $\cot\vartheta$. Indeed, decomposition into a solution of $(\Delta+2)h=0$ with the prescribed trace and a zero-trace remainder is orthogonal for $\Q_\vartheta$, and the zero-trace form is positive. The index-one conclusion and the three-dimensional kernel therefore give the stated eigenvalue identities. Rotational symmetry follows from Lima--Menezes \cite[Theorem~C]{LMcaps}.
\end{proof}

\appendix
\section{Stream coordinates on the model catenoid}
\label{sec:model}

\begin{proposition}\label[proposition]{prop:model}
The parametrisation \cref{eq:critical-param} is a properly embedded
free-boundary minimal annulus.  Its stream coordinate, with
$\beta(0)=0$, is
\begin{equation}\label{eq:model-beta}
 \beta(s)=c^2\left(\frac{s}{2}+\frac{\sinh 2s}{4}\right).
\end{equation}
In this coordinate,
\begin{equation}\label{eq:model-q}
 R=c\cosh s,\qquad H=cs,\qquad
 t=\frac{s}{\cosh s},\qquad
 t_\beta=\frac{1-s\tanh s}{c^2\cosh^3s}.
\end{equation}
\end{proposition}

\begin{proof}
The function $s\mapsto s\tanh s$ is strictly increasing for $s>0$,
starts at zero, and tends to infinity.  Hence \cref{eq:T} has a
unique positive solution.
Writing $e_r=(\cos\alpha,\sin\alpha,0)$ and $e=(0,0,1)$, the
parametrisation is
\[
 X=c\cosh s\,e_r+cs\,e.
\]
Its derivatives satisfy
\[
 X_s=c\sinh s\,e_r+ce,
 \qquad X_\alpha=c\cosh s\,e_\theta,
 \qquad |X_s|^2=|X_\alpha|^2=c^2\cosh^2s,
 \qquad X_s\cdot X_\alpha=0.
\]
Also $X_{ss}+X_{\alpha\alpha}=0$, proving minimality.
Since $H=cs$ is strictly increasing and $\alpha$ has period $2\pi$,
the parametrisation is embedded.
The identity $T\tanh T=1$ gives
\[
 \cosh^2T+T^2=T^2\cosh^2T.
\]
Therefore $|X(\pm T,\alpha)|=1$, and
$|X(s,\alpha)|<1$ for $|s|<T$ because
$\cosh^2s+s^2$ is strictly increasing for $s>0$.
At $s=T$,
\[
 \eta=\frac{X_s}{c\cosh T} =\tanh T\,e_r+\sech T\,e =X(T,\alpha),
\]
and at $s=-T$ the outward conormal is $-X_s/(c\cosh T)=X(-T,\alpha)$. Thus the boundary is free.

The metric is $g=R^2(ds^2+d\alpha^2)$, so $G=R^{-2}$. Choose the sign of the stream coordinate so that $d\beta=R^2ds$; integration gives \cref{eq:model-beta}. A normal with positive support is
\[
 N=\sech s\,e_r-\tanh s\,e,
 \qquad \xi=X\cdot N=c(1-s\tanh s).
\]
Differentiating $t=s/\cosh s$ and dividing by $\beta_s=c^2\cosh^2s$ gives \cref{eq:model-q}. The numerator is positive for $|s|<T$ and zero at $s=\pm T$. It also agrees with $\xi/(R^5G)$ in \cref{eq:q-support}.
\end{proof}

\section{The explicit Bernoulli solution and its homogeneous extension} \label{sec:explicit}

\begin{proposition}\label[proposition]{prop:explicit}
For every $e\in\Sph^2$, the domain and function in \cref{eq:classified-domain,eq:classified-function} solve \cref{eq:OEP}. Their dual surface map $X=\nabla_0u+up$ is \cref{eq:catenoid}.
\end{proposition}

\begin{proof}
For $s(p)=p\cdot e$,
\[
\nabla_0s=e-sp,\qquad |\nabla_0s|^2=1-s^2,\qquad \Delta_0s=-2s.
\]
Thus a function depending only on $s$ satisfies
\begin{equation}\label{eq:axisymmetric-Laplacian}
\Delta_0f(s)=(1-s^2)f''(s)-2sf'(s).
\end{equation}
For $f(s)=c(1-s\,\artanh s)$,
\begin{equation}\label{eq:f-derivatives}
f'(s)=-c\left(\artanh s+\frac{s}{1-s^2}\right),
\qquad
f''(s)=-\frac{2c}{(1-s^2)^2}.
\end{equation}
Substitution gives $(\Delta_0+2)f=0$.

The function $s\,\artanh s$ is even and strictly increasing for $s>0$, and
\[
s_*\artanh s_*=T\tanh T=1.
\]
Hence $f>0$ for $|s|<s_*$ and $f=0$ at $s=\pm s_*$. At either endpoint,
\[
|\nabla_0f| =\sqrt{1-s_*^2}\,|f'(s_*)| =\frac{c}{s_*\sqrt{1-s_*^2}}=1.
\]
This verifies all the boundary conditions.

Finally, writing $p=\sqrt{1-s^2}\,e_r+s e$, we obtain
\[
\nabla_0f+fp=f'e+(f-sf')p, \qquad f-sf'=\frac{c}{1-s^2}.
\]
Consequently,
\[
X(p) =\frac{c}{\sqrt{1-s^2}}\,e_r-c\,\artanh s\,e.
\]
The substitution $v=\artanh s$ gives \cref{eq:catenoid}. Its free-boundary normalisation agrees with \cref{prop:model}.
\end{proof}

\section{Angular period and finite covers}
\label{sec:period}

In this appendix we make some remarks on why embeddedness is essential for the spectral characterisation. In particular, we consider immersed annuli which are $k$-fold covers of embedded annuli.

\begin{proposition}[Finite covers]
\label[proposition]{prop:cover-index}
Let $\Sigma\subset\overline{\B^3}$ be a smooth properly
embedded free-boundary minimal annulus, and let
$\varpi:\widetilde\Sigma\to\Sigma$ be a connected smooth
$k$-sheeted covering, where $k\geq1$.
Equip $\widetilde\Sigma$ with the pullback of the induced
metric on $\Sigma$. Then the modified energy form
\[
\Q_{\widetilde\Sigma}(f)
=
\int_{\widetilde\Sigma}|\nabla f|^2\dA
-\int_{\partial\widetilde\Sigma}f^2\ds,
\qquad f\in H^1(\widetilde\Sigma),
\]
satisfies
\[
\ind\Q_{\widetilde\Sigma}=2k-1.
\]
In particular, this formula holds for the connected
$k$-fold cover of the critical catenoid.
\end{proposition}

\begin{proof}
The stream coordinates on $\Sigma$ lift to coordinates
on $\widetilde\Sigma$ identifying it with
\[
C_k=[b_-,b_+]\times\bigl(\R/(2\pi k\mathbb Z)\bigr).
\]
In these coordinates, the covering projection is $(\beta,[\alpha]_{2\pi k})
\mapsto (\beta,[\alpha]_{2\pi}).$ We use the same notation for the lifted functions $R$ and $t=h/R$. The form identity in
\cref{prop:form-identity} lifts to
\[
\Q_{\widetilde\Sigma}(Rv) = \mathfrak q_a(v) := \int_{C_k} \left(a v_\beta^2+v_\alpha^2-v^2\right)\dbda, \qquad a=R^4.
\]
Since multiplication by $R$ is an isomorphism of the corresponding $H^1$ spaces, the two forms have the same index.

The lifted function $t$ satisfies
\[
\Aop t=0,\qquad t_\beta=0\quad\text{on }\partial C_k,\qquad t_\beta>0\quad\text{in }C_k^\circ.
\]
The intertwining identity and the positive Dirichlet comparison in \cref{lem:intertwining,lem:positive-Dirichlet} apply unchanged on $C_k$. Consequently, if $\Aop v=\lambda v$ with $\lambda<0$ and Neumann boundary data, then
\[
\Bop(v_\beta)=\lambda v_\beta,\qquad v_\beta|_{\partial C_k}=0,
\]
and hence $v_\beta=0$.

The negative eigenspaces are therefore exactly those of $-\partial_{\alpha\alpha}-1$ on the circle of length $2\pi k$. Its negative modes are
\[
1,\qquad \cos\frac{j\alpha}{k},\quad \sin\frac{j\alpha}{k}, \qquad 1\leq j\leq k-1,
\]
with eigenvalues $-1$ and $j^2/k^2-1$, respectively. Each is also a Neumann eigenfunction of $\Aop$. Their total real dimension is $1+2(k-1)=2k-1,$ which proves the assertion.
\end{proof}

\begin{remark}
The index is taken over all functions in $H^1(\widetilde\Sigma)$, not only those descending to $\Sigma$, and is not the Morse index of the area functional. The additional negative angular modes above do not descend to the original annulus. Thus, the degree-one angular projection in \cref{prop:angle} is essential to the index-one conclusion for embedded annuli.
\end{remark}

\bibliographystyle{amsalpha}
\bibliography{main}

\end{document}